\documentclass[twocolumn]{autart}

\usepackage[top=.75in, bottom=.85in, left=.65in, right=.65in,paperwidth=8.5in,paperheight=11in]{geometry}

\usepackage[noadjust]{cite}
\usepackage{mathrsfs}
\usepackage{mathtools}
\usepackage{enumitem}
\usepackage{amsmath,amssymb}
\usepackage{graphicx}
\usepackage{epstopdf,cite}
\usepackage{url}
\usepackage{bm}
\usepackage{multirow,booktabs,multicol}
\usepackage{mathtools}
\usepackage{setspace}
\usepackage{makecell}
\usepackage{pifont}
\usepackage{graphicx,color}
\usepackage{subfigure,balance}
\usepackage{comment}
\usepackage[english]{babel}
\usepackage{tikz}
\usepackage[justification=centering]{caption}
\newcommand{\tr}{{{\mathsf T}}}

\newcommand{\method}[1]{\texttt{#1}}

\newcommand{\calH}{\mathcal{H}}

\newcommand{\Tr}{\operatorname{tr}}
\newcommand{\ECL}{$\mathtt{ECL}$}

\newtheorem{definition}{Definition}
\newtheorem{theorem}{Theorem}

\newtheorem{proposition}{Proposition}
\newtheorem{remark}{Remark}
\newtheorem{lemma}{Lemma}
\newtheorem{example}{Example}

\newtheorem{corollary}{Corollary}
\newtheorem{assumption}{Assumption}

\newenvironment{proof}{\textbf{Proof:}}{\hfill$\qed$}

\let\OLDthebibliography\thebibliography
\renewcommand\thebibliography[1]{
  \OLDthebibliography{#1}
  \setlength{\parskip}{0pt}
  \setlength{\itemsep}{3pt plus 0.3ex}
}

\makeatletter
\def\endthebibliography{%
  \def\@noitemerr{\@latex@warning{Empty `thebibliography' environment}}%
  \endlist
}
\makeatother

\begin{document}

\setlength{\abovedisplayskip}{9pt}
\setlength{\belowdisplayskip}{9pt}

	\begin{frontmatter}
		
		\title{Policy Optimization of Mixed $\mathcal{H}_2/\mathcal{H}_\infty$ Control: Benign Nonconvexity and Global Optimality}
		
		\thanks[footnoteinfo]{The work of Chih-Fan Pai, Yuto Watanabe, and Yang Zheng is supported by NSF CMMI 2320697 and NSF CAREER 2340713.}
		
		\author[ucsd]{Chih-Fan Pai}\ead{cpai@ucsd.edu},
		\author[ucsd]{Yuto Watanabe}\ead{y1watanabe@ucsd.edu},
		\author[Peking]{Yujie Tang}\ead{yujietang@pku.edu.cn},
		\author[ucsd]{Yang Zheng}\ead{zhengy@ucsd.edu}
		
		\address[ucsd]{Department of Electrical and Computer Engineering, University of California San Diego, CA 92093.} \vspace{-.1in}
		\address[Peking]{School of Advanced Manufacturing and Robotics, Peking University} \vspace{-.1in} \vspace{-.2in}
		

		\begin{abstract}
		{Mixed $\mathcal{H}_2/\mathcal{H}_\infty$ control balances performance and robustness by minimizing an $\mathcal{H}_2$ cost bound subject to an $\mathcal{H}_\infty$ constraint. However, classical Riccati/LMI solutions offer limited insight into the nonconvex optimization landscape and do not readily scale to large-scale or data-driven settings.
        In this paper, we revisit mixed $\mathcal{H}_2/\mathcal{H}_\infty$ control from a modern policy optimization viewpoint, including the general two-channel and single-channel cases. 
        One central result is that both cases enjoy a benign nonconvex structure: \textit{every stationary point is globally optimal}. 
        We characterize the $\mathcal{H}_\infty$-constrained feasible set, which is open, path-connected, with boundary given exactly by policies saturating the $\mathcal{H}_\infty$ constraint. We also show that the mixed objective is real analytic in the interior with explicit gradient formulas.
        Our key analysis builds on an Extended Convex Lifting ($\mathtt{ECL}$) framework that bridges nonconvex policy optimization and convex reformulations.
        The $\mathtt{ECL}$ constructions rely on non-strict Riccati inequalities that allow us to characterize global optimality.
        These insights reveal hidden convexity in mixed $\mathcal{H}_2/\mathcal{H}_\infty$ control and facilitate the design of scalable policy iteration methods in large-scale settings. 
        } \vspace{-2mm}
		\end{abstract}
		
	\end{frontmatter}

\section{Introduction}

Performance and robustness are two central objectives in control design: $\calH_2$ control optimizes average performance \cite{kalman1960contributions}, whereas $\calH_\infty$ control guarantees safety against worst-case scenarios \cite{doyle1988state}.
Mixed $\calH_2/\calH_\infty$ control provides a principled framework to balance the two, leading to a variety of formulations studied across decades  \cite{glover1988state,bernstein1989lqg,khargonekar1991mixed,kaminer1993mixed,apkarian2008mixed}.
A particularly influential formulation designs a stabilizing controller that minimizes an $\calH_2$ cost bound subject to an $\calH_\infty$ constraint \cite{bernstein1989lqg,khargonekar1991mixed}.
Classical solutions based on coupled Riccati equations \cite{bernstein1989lqg} or linear matrix inequalities (LMIs) \cite{khargonekar1991mixed} are well established,~but they provide little understanding of the underlying optimization landscape. 
Moreover, these methods are inherently model-based and scale poorly with system dimension, limiting their applicability in large-scale or data-driven settings.

In contrast, policy optimization has emerged as a promising alternative for controller design \cite{hu2023toward,Talebi2024geometry}, inspired by the success of reinforcement learning in sequential decision-making and continuous control tasks \cite{recht2019tour}.
Despite the nonconvexity of policy spaces, recent studies have revealed benign landscapes in various control problems such as stabilization \cite{perdomo2021stabilizing,zhao2022sample}, linear quadratic regulation (LQR) \cite{fazel2018global,mohammadi2021convergence,watanabe2025revisiting}, linear quadratic Gaussian (LQG) \cite{tang2021analysis,zheng2022escaping,duan2022optimization}, and dynamic filtering \cite{umenberger2022globally}. For these control problems, stationary points can be globally optimal, and gradient-based methods can achieve global convergence under mild assumptions.

A key insight behind these benign landscape results is that many nonconvex control problems can be reformulated as convex ones via appropriate changes of variables.
In particular, LMI-based synthesis methods, despite involving auxiliary Lyapunov variables, provide a lens for analyzing the geometry of policy optimization \cite{sun2021learning,mohammadi2021convergence,tang2021analysis,guo2022global,umenberger2022globally,hu2022connectivity}.
This perspective has been formalized in the Extended Convex Lifting (\ECL{}) framework \cite{zheng2024benign}, which systematically bridges classical convex reformulations and modern nonconvex policy optimization.
The \ECL{} framework is broadly applicable, encompassing state-feedback and output-feedback $\calH_2$ and $\calH_\infty$ control as well as distributed control.

In this work, we revisit the classical mixed $\calH_2/\calH_\infty$ control design from a modern policy optimization perspective.
Building on the classical formulations of \cite{bernstein1989lqg,khargonekar1991mixed}, we analyze both the general two-channel case and its single-channel specialization, and provide a systematic study of the associated nonconvex optimization landscapes.
Our results reveal \textit{hidden convexity} and establish \textit{the absence of spurious stationary points}. Beyond new theoretical insights into nonconvex optimization in modern control, these results will facilitate the mixed $\calH_2/\calH_\infty$ design in large-scale and/or model-free data-driven settings. 

\subsection{Our contributions}
This paper presents a systematic study of mixed $\calH_2/\calH_\infty$ state-feedback control through the lens of modern nonconvex optimization.
Technically, our main  contributions are:
\begin{enumerate}
    \item \textbf{Basic landscape properties.} We analyze the \textit{geometry} of the nonconvex feasible set. It is known that this set is open and path-connected (Lemma~\ref{lemma:path-connectedness}). We further precisely characterize its boundary as the set of policies that exactly saturate the $\calH_\infty$ constraint (Theorem~\ref{prop:closure} and Corollary~\ref{corollary:boundary-set}).
    We also examine the \textit{landscape} of the mixed cost function. This cost function is real analytic in the interior, and continuous on the closure of its domain  (Theorem~\ref{theorem:analytical-functions} and Lemma~\ref{thm:continuity-mix-cost}). Thus, the mixed cost function is smooth, and we provide its explicit gradient formulas (Lemmas~\ref{theorem:policy-gradient-mixed} and \ref{theorem:policy-gradient-mixed-signle-channel}). These results underpin our analysis of stationarity, global optimality, and solvability of mixed $\calH_2/\calH_\infty$ control.
    
    \item \textbf{Global optimality of stationary points.} We investigate the \textit{global optimality} of the two-channel mixed control. This includes the single-channel setting in \cite{zhang2021policy} as a special case. Despite its nonconvexity, we establish that no spurious stationary points exist (Theorem~\ref{thm:mix-stationary-global-optimality}). This property, along with the gradient expressions, recovers the classical optimality conditions (Corollaries~\ref{thm:mixed-2ch-optimality-conditions} and \ref{thm:mixed-1ch-optimality-conditions}). We further analyze existence and uniqueness, showing that the single-channel case always admits a unique stationary point (Theorem~\ref{theorem:unique-stationary-point}), whereas the two-channel case may not (Fact~\ref{fact:solvability-2ch-mix}). Nevertheless, we prove that a stationary point exists when the robustness constraint is sufficiently relaxed (Theorem~\ref{prop:existence-2ch-large-beta}).
    
    \item \textbf{Analysis techniques via \ECL{}}. We explicitly construct an extended convex lifting (\ECL{}) for the two-channel mixed control (Theorem~\ref{proposition:mixed-lifting}). While relying on classical LMI techniques, our convex lifting construction is non-trivial since we need to employ \textit{non-strict} Riccati inequalities and LMIs, in contrast to classical suboptimal controller synthesis based on \textit{strict} inequalities \cite{khargonekar1991mixed}. 
    This key distinction enables the \ECL{} framework to certify the global optimality of stationary points over the entire feasible set. Moreover, the resulting convex reformulation (Theorem~\ref{thm:mix-convex-reformulation}) not only preserves the optimal value of the original nonconvex problem, but also guarantees solvability when incorporating boundary policies (Proposition~\ref{thm:solvability}).
\end{enumerate}

\vspace{1mm}

\subsection{Related work} \label{subsec:related-work}

\noindent\textbf{Mixed $\calH_2/\calH_\infty$ control.} 
The study of mixed $\calH_2/\calH_\infty$ control dates back to the seminal works of \cite{bernstein1989lqg,khargonekar1991mixed}, which formulated the problem of minimizing an $\calH_2$ cost bound subject to an $\calH_\infty$ constraint.
Solvability of this formulation was first established in \cite{bernstein1989lqg} via coupled Riccati equations, and later \cite{khargonekar1991mixed} proposed a more tractable approach for suboptimal controller synthesis.
A special case was later shown in \cite{mustafa1989relations} to coincide with the maximum entropy $\calH_\infty$ control \cite{glover1989derivation}.
A Nash game formulation was also developed in \cite{limebeer1994nash}, framing the mixed design as a two-player game with the optimal controller characterized by coupled Riccati equations.
Beyond analytical methods, a non-smooth optimization approach was proposed in \cite{apkarian2008mixed}, ensuring convergence to a stationary point controller via a proximity control algorithm.
Furthermore, LMI-based methods have also been developed in \cite{scherer1997multiobjective}, providing numerical solutions through semidefinite programming.
More recently, \cite{zhang2021policy} studied policy gradient methods for the single-channel case, showing that the iterates implicitly preserve the $\calH_\infty$ constraint throughout optimization and converge to globally optimal policies.
Although their analysis proves global optimality of stationary points, it applies only to a special single-channel case of the broader formulation in \cite{khargonekar1991mixed}, and relies on a game-theoretic argument that may be less accessible to those unfamiliar with dynamic game theory \cite{bacsar2008h}.

\noindent\textbf{Policy optimization in control.}
Recent years have seen growing interest in direct policy search for controlling dynamical systems \cite{recht2019tour,hu2023toward}.
For classical LQR, the cost is coercive and gradient dominant, yielding global convergence of policy gradient methods \cite{fazel2018global,mohammadi2021convergence,watanabe2025revisiting}.
Similar benign landscapes have been established for Markovian jump LQR \cite{jansch2022policy}, distributed LQR \cite{li2022distributed,furieri2020learning}, LQ games \cite{zhang2019policy,watanabe2025semidefinite}, and mixed $\calH_2/\calH_\infty$ design \cite{zhang2021policy}.
For nonsmooth $\calH_\infty$ optimization, global convergence is established for state feedback \cite{guo2022global}, and convergence to Goldstein stationary points is shown for static output feedback \cite{guo2023complexity}.
Partial observation and dynamic output feedback introduce more intricate landscapes: in LQG, controllable and observable stationary points are globally optimal though saddle points may exist \cite{tang2023analysis,zheng2023benign}, motivating saddle-escaping \cite{zheng2022escaping} and Riemannian gradient methods \cite{kraisler2024output}.
For dynamic output feedback $\calH_\infty$, stabilizing controllers are nonconvex and may form a disconnected set \cite{hu2022connectivity}, but all nondegenerate Clarke stationary points are globally optimal \cite{zheng2023benign}. We refer to \cite{hu2023toward,Talebi2024geometry} for two recent surveys.

\noindent\textbf{Convex lifting for nonconvex control.}
A growing line of work has leveraged convex reformulations to analyze nonconvex control landscapes \cite{mohammadi2021convergence,watanabe2025revisiting,sun2021learning,guo2022global,umenberger2022globally,tang2021analysis}.
Many classical control problems admit LMI-based reformulations \cite{scherer1997multiobjective}, which explain the benign geometry observed in policy optimization.
For example, convex reparameterizations yield global convergence in continuous-time LQR \cite{mohammadi2021convergence,watanabe2025revisiting}, and certify global optimality of Clarke stationary points in discrete-time $\mathcal{H}_\infty$ control \cite{guo2022global}. 
Extensions to output feedback settings include convex lifting strategy for dynamic estimation \cite{umenberger2022globally} and global optimality results for $\calH_\infty$ control \cite{tang2023analysis}.
Most recently, the \ECL{} framework \cite{zheng2024benign} unifies these perspectives, accommodating static and dynamic policies, smooth and nonsmooth costs, and distinguishing degenerate from nondegenerate policies based on non-strict LMIs.

\subsection{Paper outline}
The rest of this paper is organized as follows.
Section~\ref{section:preliminaries} formulates the mixed $\calH_2/\calH_\infty$ policy optimization. 
Section~\ref{section:opt-landscape} examines the geometry of the feasible set and the structural properties of the cost function.
Section~\ref{section:no-spurious-stationary} establishes our main global optimality result, along with characterizations of optimality conditions and the existence and uniqueness of stationary points.
Section~\ref{section:ECL} presents the \ECL{} framework underlying our analysis and develops a tailored \ECL{} for mixed control.
Section~\ref{section:experiments} reports numerical results, and Section~\ref{sec:conclusions} concludes the paper.
Technical proofs are provided in the appendices.

\vspace{2pt}
\noindent \textbf{Notations.} 
We denote the set of $k\times k$ real symmetric matrices by $\mathbb{S}^k$.
For $M_1, M_2\!\in\! \mathbb{S}^k$, we write $M_1\prec (\preceq) M_2$ and $M_2 \succ (\succeq) M_1$ if $M_2-M_1$ is positive (semi)definite. 
The Frobenius norm for matrices is denoted by $\|\cdot\|$.
We use $I_n$ and $0_{m\times n}$ for the $n\times n$ identity and $m\times n$ zero matrices, respectively, omitting their subscripts when clear. For a subset $S$ of a topological space, $\operatorname{int}(S)$ denote its interior and $\operatorname{cl}(S)$ its closure. 

\section{Preliminaries} \label{section:preliminaries}

This section introduces the mixed $\calH_2/\calH_\infty$ control problem and presents our problem statement.

\subsection{System dynamics and robustness}
Consider the continuous-time linear dynamical system
\begin{equation}\label{eq:Dynamics-state}
    \begin{aligned}
    \dot{x}(t) &= Ax(t)+Bu(t)+ B_ww(t), 
    \end{aligned}
\end{equation}
where $x(t) \in \mathbb{R}^n$ is the state, $u(t)\in \mathbb{R}^m$ is the control input, $w(t) \in \mathbb{R}^n$ is the disturbance.
We focus on static state feedback policies of the form $u(t) = Kx(t)$, where $K\in \mathbb{R}^{m \times n}$. The set of stabilizing policies is defined as
\begin{equation} \label{eq:stabilizing-set}
    \mathcal{K}=
    \left\{K \in \mathbb{R}^{m \times n} \mid 
    A\!+\!BK \text{ is Hurwitz}\right\}.
\end{equation}
We henceforth denote $A_K:=A\!+\!BK$ and $W := B_wB_w^\tr$.

In standard LQR, the disturbance $w(t)$ is modeled as zero-mean white Gaussian noise with identity covariance, i.e., $\mathbb{E}[w(t)w(\tau)]=\delta(t-\tau)I_n$. The objective is to find a stabilizing policy $K\in\mathcal{K}$ that minimizes the averaged cost
$
    \lim_{T \to \infty}\mathbb{E}\left[\frac{1}{T}\int_0^{T}
        x(t)^\tr Q_2 x(t) \!+\! u(t)^\tr R_2 u(t)\,dt\right], 
$
where $Q_2\!\succeq\!0$ and $R_2\succ0$ are performance weight matrices. It is known that LQR can be equivalently cast as $\mathcal{H}_2$ optimal control \cite{zhou1996robust}: $\min_{K \in \mathcal{K}} \! \|\mathbf{T}_{2}(K)\|_{\mathcal{H}_2}^2$, where $z_2$ defines the $\calH_2$ performance output and $\mathbf{T}_{2}(K)$ denotes the transfer function from $w$ to $z_2$ as follows
\begin{align}
    z_2 & \!:=\! \begin{bmatrix}
        Q_2^{1/2}x \\ R_2^{1/2}u
    \end{bmatrix},\
    \mathbf{T}_{2}(K) \!=\! \begin{bmatrix}
        Q_2^{1/2} \\ R_2^{1/2}K
    \end{bmatrix} (sI - A_K)^{-1}B_w. \label{eq:transfer-function-Tz2w}
\end{align} 

On the other hand, $\mathcal{H}_\infty$ robust control treats the disturbance $w(t)$ as an adversarial input with bounded energy. The aim is to design a stabilizing policy $K\in\mathcal{K}$ that minimizes the worst-case cost
$
    \sup \! \int_0^{\infty}
        x(t)^\tr Q_\infty x(t) \!+\! u(t)^\tr R_\infty u(t) dt
$
subject to $\|w\|_{\ell_2}^2\!=\!\int_0^\infty w^\tr(t)w(t)dt\!\leq\!1$ and $x(0)\!=\!0$, where $Q_\infty \!\succ\! 0$ and $R_\infty \!\succ\! 0$. 
This problem is equivalent to $\mathcal{H}_\infty$ optimal control \cite{zhou1996robust}: $\inf_{K \in \mathcal{K}}\! \|\mathbf{T}_{\infty}(K)\|_{\mathcal{H}_\infty}$, where $z_\infty$ defines the $\calH_\infty$ performance output and $\mathbf{T}_{\infty}(K)$ denotes the transfer function from $w$ to $z_\infty$ as follows
\begin{align}
    z_\infty &\!:=\!\begin{bmatrix}
        Q_\infty^{1/2}x \\ R_\infty^{1/2}u
    \end{bmatrix},\
    \mathbf{T}_{\infty}(K) \!=\! \begin{bmatrix}
        Q_\infty^{1/2} \\ R_\infty^{1/2}K
    \end{bmatrix} (sI \!-\! A_K)^{-1}B_w. \label{eq:transfer-function-Tziw}
\end{align} 

One may also consider a mixed $\calH_2/\calH_\infty$ design
\begin{equation} \label{eq:LQR-hinf}
\begin{aligned}
    \inf_{K \in \mathcal{K}}&\quad  \|\mathbf{T}_{2}(K)\|_{\mathcal{H}_2}^2,\\
    \text{subject to}& \quad  \|\mathbf{T}_{\infty}(K)\|_{\mathcal{H}_\infty}<\beta,
\end{aligned}
\end{equation}
where $\beta$ is a prescribed bound. 
Here, the $\calH_2$ channel captures nominal performance, while the $\calH_\infty$ channel enforces system robustness.
It is worth noting that these two channels may be distinct \cite{bernstein1989lqg,khargonekar1991mixed} or identical \cite{mustafa1989relations,zhang2021policy}.
We define the feasible set of $\mathcal{H}_\infty$-constrained stabilizing policies as 
\begin{equation} \label{eq:mix-feasible-set}
    \mathcal{K}_\beta \coloneqq
    \left\{K \in \mathcal{K} \mid \|\mathbf{T}_{\infty}(K)\|_{\mathcal{H}_\infty} \!<\! \beta\right\}.
\end{equation}
Note that a strict $\mathcal{H}_\infty$ bound is commonly used to ensure a well-defined stabilizing Riccati solution \cite{khargonekar1991mixed,mustafa1989relations,zhang2021policy}.

\subsection{Policy optimization for mixed $\mathcal{H}_2/\mathcal{H}_\infty$ design} \label{subsec:mixed-PO}

Despite simple formulation, problem \eqref{eq:LQR-hinf} is challenging to solve \cite{rotea1991h2,megretski1994order}.
A classical approach from \cite{bernstein1989lqg} is to minimize an upper bound on the $\mathcal{H}_2$ cost over the feasible set $\mathcal{K}_\beta$.

Recall that the $\mathcal{H}_2$ cost for any stabilizing policy $K\in\mathcal{K}$ is $\|\mathbf{T}_{2}(K)\|_{\mathcal{H}_2}^2 = \Tr((Q_2+K^\tr R_2 K)\hat{X}_K)$, where $\hat{X}_K$ is the unique solution to the Lyapunov equation
\begin{equation} \label{eq:H2-Lyapunov}
    A_K\hat{X}_K + \hat{X}_KA_K^\tr + W = 0.
\end{equation}
To enforce the $\mathcal{H}_\infty$ constraint in \eqref{eq:mix-feasible-set}, we instead consider the stabilizing solution $X_K$ to the Riccati equation \cite{khargonekar1991mixed,glover1988state,bernstein1989lqg} 
\begin{equation} \label{eq:mix-2ch-Riccati-XK}
    A_KX_K + X_KA_K^\tr + \beta^{-2}X_K S_K X_K + W = 0,
\end{equation}
where $S_K\!:=\!Q_\infty\!+\!K^\tr R_\infty K$. 
By the bounded real lemma (see Lemma~\ref{lem:BRL-X}), a policy $K\in\mathcal{K}_\beta$ if and only if \eqref{eq:mix-2ch-Riccati-XK} admits a unique stabilizing solution $X_K \succeq 0$ such that the matrix $A_K+\beta^{-2}X_KS_K$ is Hurwitz.
Subtracting \eqref{eq:H2-Lyapunov} from \eqref{eq:mix-2ch-Riccati-XK} gives  
\begin{align*}
    A_K(X_K-\hat{X}_K) &+ (X_K-\hat{X}_K)A_K^\tr + \beta^{-2}X_KS_K X_K=0.
\end{align*}
Since $A_K$ is Hurwitz and the last term is positive semidefinite, it follows that $X_K \!-\! \hat{X}_K \!\succeq\! 0$, and hence the $\mathcal{H}_2$ cost admits an upper bound 
$
    \|\mathbf{T}_{2}(K)\|_{\mathcal{H}_2}^2 \leq \Tr((Q_2+K^\tr R_2 K)X_K)
$ for all $K \in \mathcal{K}_\beta$.
Therefore, we define the mixed cost
\begin{equation} \label{eq:mix-cost-2ch}
    J_{\mathrm{mix}}(K) := \Tr((Q_2+K^\tr R_2 K)X_K), \quad \forall K\in\mathcal{K}_\beta,
\end{equation}
where $X_K$ is the unique stabilizing solution to \eqref{eq:mix-2ch-Riccati-XK}. 

In sum, we consider the following mixed design \cite{khargonekar1991mixed,glover1988state,bernstein1989lqg}
\begin{equation} \label{eq:mix-2ch-problem}
    \inf_{K \in \mathcal{K}_{\beta}}\  J_{\mathrm{mix}}(K).  
\end{equation}
Unlike the $\mathcal{H}_2$ cost, $J_{\mathrm{mix}}$ is evaluated using the Riccati solution $X_K$ from \eqref{eq:mix-2ch-Riccati-XK} rather than the Lyapunov solution $\hat{X}_K$ from \eqref{eq:H2-Lyapunov}.
Since the cost and constraint in \eqref{eq:mix-2ch-problem} are in general defined by distinct signals $z_2$ and $z_\infty$, we refer to this formulation as the \textit{two-channel} mixed $\mathcal{H}_2$/$\mathcal{H}_\infty$ design.

We make the following standard assumption throughout. 
\begin{assumption} \label{assumption:1}
    $(A, B)$ is stabilizable and $(Q_2^{1/2}, A)$ is detectable. In addition, the robustness parameter satisfies $\beta>\beta^\ast:=\inf_{K\in\mathcal{K}}\|\mathbf{T}_{\infty}(K)\|_{\calH_\infty}$.
\end{assumption}

We make another assumption for analysis \cite{fatkhullin2021optimizing,mohammadi2021convergence}.
\begin{assumption} \label{assumption:positive-definiteness}
    The matrix $Q_2$ is positive semidefinite. The matrices $W$, $Q_\infty$, $R_2$, and $R_\infty$ are positive definite.
\end{assumption}

\begin{remark}[Single-channel case]
    When the $\mathcal{H}_2$ and $\mathcal{H}_\infty$ channels share the same performance output, $J_{\mathrm{mix}}$ in \eqref{eq:mix-cost-2ch} simplifies. Specifically, let $Q\!:=\!Q_2\!=\!Q_\infty$ and $R\!:=\!R_2\!=\!R_\infty$ in \eqref{eq:transfer-function-Tz2w} and \eqref{eq:transfer-function-Tziw}, so that $z_2=z_\infty$.
    Under Assumption~\ref{assumption:1},
    \begin{subequations} \label{eq:mix-1ch-cost-2}
    \begin{equation} \label{eq:mix-1ch-cost-equivalence}
        {J}_{\mathrm{mix}}(K) = \operatorname{tr}({P}_KW), \quad \forall K\in\mathcal{K}_\beta,
    \end{equation}
    where ${P}_K$ is the stabilizing solution to the Riccati equation
    \begin{align}  \label{eq:mix-1ch-Riccati-PK}
        A_K^\tr{P}_K\!+\!{P}_KA_K\!+\!\beta^{-2}{P}_KW{P}_K\!+\!Q\!+\!K^\tr R K&\!=\!0
    \end{align}
    with $A_K\!+\!\beta^{-2}WP_K$ Hurwitz.
    \end{subequations}
   We give further details in Appendix~\ref{proof:alternative-cost-single-channel}.  
    As we will see, this reformulation yields a simpler gradient formula and a closed-form optimum. \hfill $\qed$
\end{remark}

\subsection{Problem statement} \label{subsec:problem-statement}

Classical solutions to mixed $\mathcal{H}_2/\mathcal{H}_\infty$ design primarily rely on Riccati equations or LMIs. While effective for small- to medium-scale systems, these techniques offer limited insight into the optimization landscape, and may face challenges in high-dimensional or data-driven settings.

In this work, we revisit the mixed \( \mathcal{H}_2/\mathcal{H}_\infty \) control in \eqref{eq:mix-2ch-problem} from a nonconvex optimization perspective. Our goal is to uncover key structural properties and lay a theoretical foundation for policy optimization, focusing on three aspects:

\begin{enumerate}
    \item \textbf{Geometry of the optimization landscape.} We study the geometry of the $\mathcal{H}_\infty$-constrained feasible domain \( \mathcal{K}_\beta \) in \eqref{eq:mix-feasible-set}, and analyze key properties of the mixed cost function in \eqref{eq:mix-cost-2ch}, including continuity, analyticity, and gradient expressions.

    \item \textbf{Global optimality of stationary points.} We investigate whether problem \eqref{eq:mix-2ch-problem} admits spurious stationary points.
    We further derive the optimality conditions, and analyze the existence and uniqueness of stationary points. 

    \item \textbf{Analysis via the \ECL{} framework.} We explore if the hidden convexity in problem \eqref{eq:mix-2ch-problem} can be revealed via an appropriate convex lifting. In particular, we ask if an extended convex lifting (\ECL{}) can be constructed to certify the global optimality of stationary points.
\end{enumerate}

Collectively, our results provide a renewed understanding of mixed $\calH_2/\calH_\infty$ control through nonconvex optimization and convex lifting, and offer guidance for the design of principled policy optimization algorithms in large-scale or model-free settings.

\section{Basic Landscape Properties} \label{section:opt-landscape}

In this section, we investigate the optimization landscape of \eqref{eq:mix-2ch-problem}, i.e., the geometry of the feasible set $\mathcal{K}_\beta$ and structural properties of the mixed cost function $J_{\mathrm{mix}}$.

\subsection{Geometry of the $\mathcal{H}_\infty$ constrained domain} \label{subsec:property-mix-feasible-set}

We start with a version of the \textit{Bounded Real Lemma}, which connects a Riccati equation (or inequality) to an upper bound on the $\calH_\infty$ norm. Here, we slightly abuse the notation for the matrices $A$, $B$, and $C$, but this should cause no confusion in the context. 

\begin{lemma}[Bounded real lemma {\cite[Corollary 13.24]{zhou1996robust}}] \label{lem:BRL-X}
Consider the transfer function $\mathbf{G}(s) = C(sI - A)^{-1}B$, with $A\in\mathbb{R}^{n\times n}$, $B\in\mathbb{R}^{n\times m}$, and $C\in\mathbb{R}^{p\times n}$. Then, for any $\beta>0$, the following statements~are~equivalent. 
\begin{enumerate}
    \item The matrix $A$ is Hurwitz and $\|\mathbf{G}(s)\|_{\calH_\infty} < \beta$.
    
    \item There exists an $X\succ 0$ satisfying the Riccati inequality
    \begin{equation} \label{eq:Riccati-Hinf-strict}
        A X + X A^\tr + \beta^{-2}X C^\tr C X + BB^\tr \prec 0. 
    \end{equation}

    \item The Riccati equation
    \begin{equation} \label{eq:Hinf-Reccati-equation}
        A X + X A^\tr + \beta^{-2}X C^\tr C X + BB^\tr = 0
    \end{equation}
    admits a unique stabilizing solution $X$ (i.e., $A + \beta^{-2}XC^\tr C$ is Hurwitz) satisfying $X\succeq0$. If $(A,B)$ is further controllable, we have $X \succ 0$.
\end{enumerate}
\end{lemma}

From \cite[Corollary 13.13]{zhou1996robust}, the stabilizing solution $X$ to \eqref{eq:Hinf-Reccati-equation} is also \textit{minimal}, i.e., $X\preceq \hat{X}$ for any other symmetric solution $\hat{X}$ to \eqref{eq:Hinf-Reccati-equation}.
This property clarifies the stabilizing solution in defining the mixed cost function \eqref{eq:mix-cost-2ch}, as it provides the tightest upper bound on the corresponding $\calH_2$ norm. 

Lemma~\ref{lem:BRL-X} plays a crucial role in understanding the feasible set and cost function of the mixed $\mathcal{H}_2/\mathcal{H}_\infty$ design \eqref{eq:mix-2ch-problem}. 
We summarize several basic properties of  $\mathcal{K}_\beta$ below. 

\begin{lemma} \label{lemma:path-connectedness}
    Suppose Assumptions~\ref{assumption:1} and \ref{assumption:positive-definiteness} hold. The set $\mathcal{K}_\beta$ in \eqref{eq:mix-feasible-set} satisfies the properties below:
    \begin{enumerate}
        \item It is always nonempty, open, and path-connected;  
        \item It is, in general, nonconvex and unbounded.
    \end{enumerate}
\end{lemma} 
\begin{proof}
    The nonemptiness of $\mathcal{K}_\beta$ follows directly from Assumption~\ref{assumption:1}.
    To establish openness, we apply Lemma~\ref{lem:BRL-X} to the transfer function $\mathbf{T}_{\infty}(K)$. This ensures that $K\in\mathcal{K}_\beta$ if and only if there exists $X\succ0$ (the positive definiteness is ensured by $W\succ0$) satisfying the strict Riccati inequality
    \begin{equation} \label{eq:BRL-Riccati-inequality-feedback-X-main}
        A_KX + XA_K^\tr + \beta^{-2}XS_K X + W \prec 0. 
    \end{equation}
    Since the inequality is strict, the same $X\succ0$ ensures that the inequality holds for all $K'$ in a sufficiently small neighborhood around $K$. Hence, $\mathcal{K}_\beta$ is open. 

    Path-connectivity can be established using a strategy similar to that in \cite[Section III A]{hu2022connectivity}. By applying the Schur complement to \eqref{eq:BRL-Riccati-inequality-feedback-X-main} and introducing the change of variables $K=YX^{-1}$, we see that $K\in\mathcal{K}_\beta$ if and only if there exist matrices $X\succ0$ and $Y$ satisfying $\mathbb{F}_{\mathrm{cvx}}\prec0$, where
    \begin{equation}\label{eq:mix-feasible-set-connectivity}
        \scalebox{0.85}{$\mathbb{F}_{\mathrm{cvx}}\!:=\!
        \begin{bmatrix}
        AX + XA^\tr + BY + Y^\tr B^\tr + W & X Q_\infty^{1/2} & Y^\tr R_\infty^{1/2} \\
        Q_\infty^{1/2} X & -\beta^2 I & 0 \\
        R_\infty^{1/2} Y & 0 & -\beta^2 I
        \end{bmatrix}
        $.}
    \end{equation}
    The above inequalities define a convex set in the lifted variables $(X,Y)$, and since the map $K=YX^{-1}$ is continuous, it follows that $\mathcal{K}_\beta$ is path-connected.

    While the nonconvexity of $\mathcal{K}_\beta$ is well-known, its unboundedness has not been discussed\footnote{The set $\mathcal{K}_\beta$ in the discrete-time is bounded due to the coercivity of the discrete-time $\calH_\infty$ function~\cite{guo2022global}.}. The example below illustrates the nonconvexity and unboundedness of $\mathcal{K}_\beta$.
\end{proof}

\begin{figure*}
    \centering
    \subfigure[Nonconvexity of $\mathcal{K}_\beta$]{%
        \includegraphics[width=0.27\textwidth]{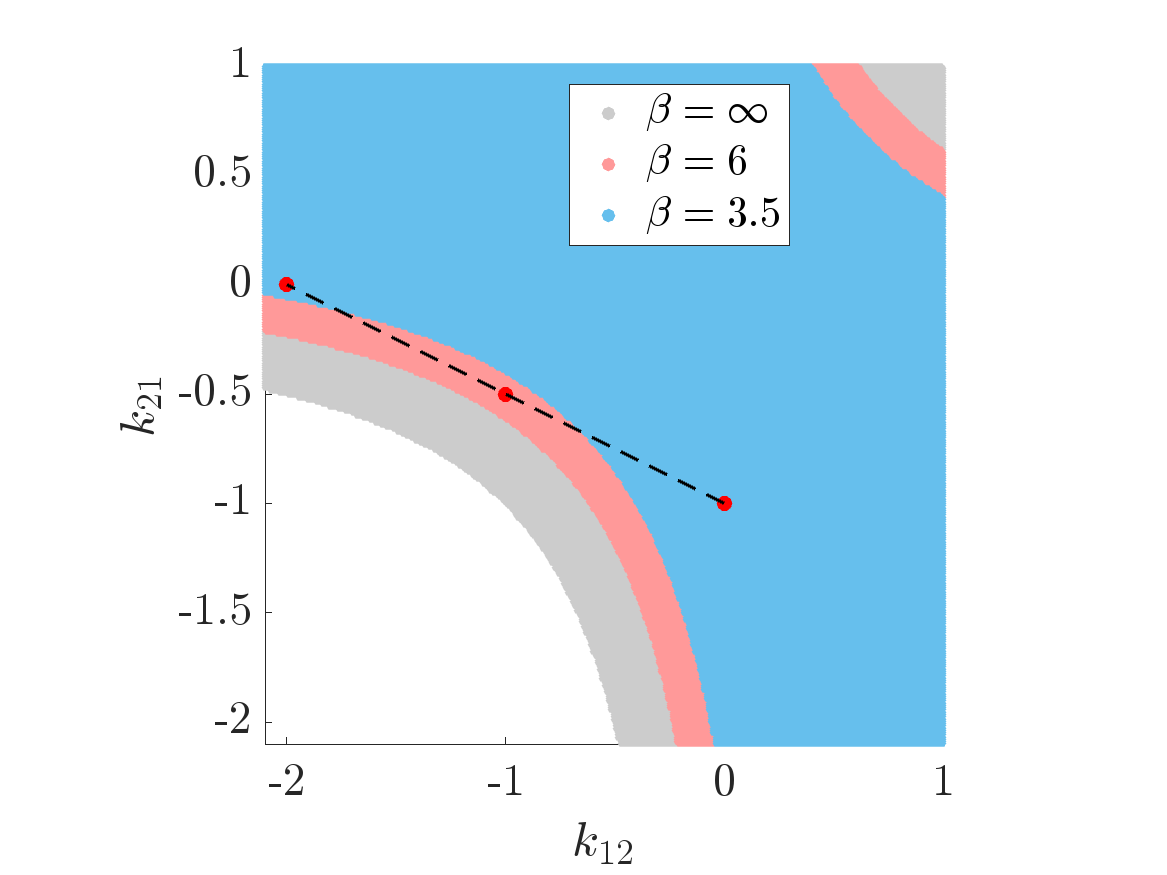}
        \label{fig:mixed-feasible-set}
    }\hspace{4mm}
    \subfigure[Nonconvex noncoercive $J_{\mathrm{mix}}$]{%
        \includegraphics[width=0.27\textwidth]{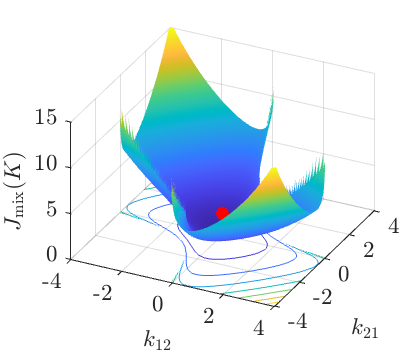}
        \label{fig:mix-cost-not-coercive}
    }\hspace{4mm}
    \subfigure[Nonconvex coercive $J_{\mathrm{LQR}}$]{%
        \includegraphics[width=0.27\textwidth]{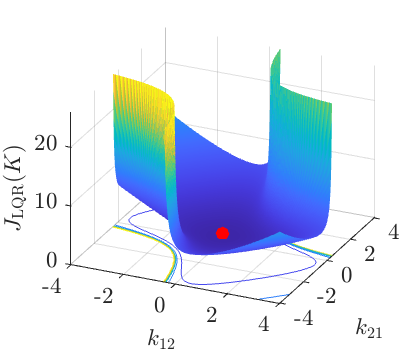}
        \label{fig:lqr-cost-coercive}
    }
    \caption{(a) Nonconvex $\mathcal{K}_{\beta}$ with $k_{11}=k_{22}=0$ for different $\beta$ in Example~\ref{example:mix-set-nonconvexity-unbounded}. 
    (b)-(c) Nonconvexity and (non)coercivity of the costs in Example~\ref{example:mix-cost-nonconvex-noncoercive}, with feasible sets $\mathcal{K}_{3.5}$ and $\mathcal{K}_{\infty}$ for (b) and (c) respectively. Red dots highlight global minima.}
    \label{fig:non-convexity-illustration}
\end{figure*}

\begin{example}
\label{example:mix-set-nonconvexity-unbounded}
    Consider
    $ 
    A=-I_2,\; B\!=\!B_w\!=\!Q_\infty\!=\!R_\infty=I_2,
    $ 
    and policies of the form $K=\left[\begin{smallmatrix}
        k_{11} & k_{12} \\ k_{21} & k_{22}
    \end{smallmatrix}\right]$ for $\mathcal{K}_\beta$ in \eqref{eq:mix-feasible-set}.
   Let $\beta=3.5$. We numerically verify that
    \[
        K_1=\begin{bmatrix}0 & 0 \\ -1 & 0 \end{bmatrix}\in\mathcal{K}_{3.5},\quad  K_2=\begin{bmatrix} 0 & -2 \\ 0 & 0 \end{bmatrix}\in\mathcal{K}_{3.5},
    \]
    but $K_3=\frac{1}{2}(K_1+K_2)\notin \mathcal{K}_\beta$.
    To see the unboundedness of $\mathcal{K}_\beta$, we analytically compute the $\calH_\infty$ norm below  
    \[
        \|\mathbf{T}_{\infty}(K)\|_{\mathcal{H}_\infty}=\frac{\sqrt{k^2+1}}{1-k}<1.1, \quad \forall K=k\!\cdot\! I_2,\ k<0.
    \]
    This implies that for any $\beta>1.1$, $\mathcal{K}_\beta$ is unbounded.
    
    Fig.~\ref{fig:mixed-feasible-set} plots $\mathcal{K}_\beta$ restricted to the subspace $k_{11}=k_{22}=0$ for $\beta=3.5$, $6$, and $\infty$ (i.e., the unconstrained stabilizing set $\mathcal{K}$). The policies $K_1$, $K_2$, and $K_3$, marked in red, demonstrate the nonconvexity for $\beta=3.5$. As expected, the feasible set $\mathcal{K}_\beta$ becomes larger as $\beta$ increases.
    \hfill $\qed$
\end{example}

We next characterize the boundary of $\mathcal{K}_\beta$. This will be useful as we later discuss the properties of $J_\mathrm{mix}$.

\begin{theorem} \label{prop:closure}
    Let $\mathrm{cl}({\mathcal{K}_\beta})$ denote the closure of $\mathcal{K}_\beta$ defined in \eqref{eq:mix-feasible-set}. Under Assumptions~\ref{assumption:1} and \ref{assumption:positive-definiteness}, we have 
    $$
    \mathrm{cl}({\mathcal{K}_\beta}) = \left\{K\in\mathcal{K} \mid \|\mathbf{T}_{\infty}(K)\|_{\mathcal{H}_\infty} \leq \beta\right\}.
    $$
\end{theorem}

Theorem~\ref{prop:closure} may seem obvious, but the argument is in fact subtle. 
For instance, consider $f(x)=x^2(x+1)$ and $C_1 = \{x\in\mathbb{R} \mid f(x)< 0 \}$, $C_2 = \{x\in\mathbb{R} \mid f(x)\leq 0 \}$.
Clearly, $C_2 \nsubseteq \mathrm{cl}(C_1)$ since $0\in C_2$ cannot be approached by any sequence in $C_1$.
In fact, Theorem~\ref{prop:closure} relies on fundamental properties of the $\mathcal{H}_\infty$ norm to handle the above issue as well as marginally stabilizing policies.
We present the full details in Appendix~\ref{appendix:proof-boundary}. 

Theorem~\ref{prop:closure} implies that the boundary of $\mathcal{K}_\beta$, denoted $\partial{\mathcal{K}_\beta} :=\mathrm{cl}(\mathcal{K}_\beta)\!\setminus\!\mathcal{K}_\beta$, is the set of stabilizing policies whose $\mathcal{H}_\infty$ norm equals $\beta$. Note that $\partial{\mathcal{K}_\beta}$ may be unbounded.

\begin{corollary} \label{corollary:boundary-set}
    Under Assumptions~\ref{assumption:1} and \ref{assumption:positive-definiteness}, the boundary of $\mathcal{K}_\beta$ satisfies    $\partial{\mathcal{K}_\beta} \subseteq \left\{K\in\mathcal{K} \mid \|\mathbf{T}_{\infty}(K)\|_{\mathcal{H}_\infty} \!=\! \beta\right\}.
    $
\end{corollary}

\subsection{Analyticity of the mixed $\mathcal{H}_2/\mathcal{H}_\infty$ cost function} \label{subsec:property-mix-cost-function}

We now examine some properties of the cost function. 
It is clear from Lemma~\ref{lemma:path-connectedness} that $J_{\mathrm{mix}}$ is nonconvex. Our next result shows that it is real analytic (and hence infinitely differentiable) on the feasible set $\mathcal{K}_\beta$.

\begin{theorem} \label{theorem:analytical-functions}
    Under Assumption~\ref{assumption:1}, the mixed $\mathcal{H}_2/\mathcal{H}_\infty$ cost function $J_{\mathrm{mix}}$ in \eqref{eq:mix-cost-2ch} is real analytic on $\mathcal{K}_\beta$.  
\end{theorem}

Note that $J_{\mathrm{mix}}$ is defined through the stabilizing solution $X_K$ of \eqref{eq:mix-2ch-Riccati-XK}, which admits no closed form. Instead, we use the Implicit Function Theorem \cite[Theorem 2.3.5]{krantz2002primer} to show that $X_K$ depends analytically on $K\in\mathcal{K}_\beta$. Thus, the cost $J_{\mathrm{mix}}$ is real analytic. Proof details are provided in Appendix~\ref{appendix:proof-analyticity}. 

Since $\mathcal{K}_\beta$ is open, the optimal value of problem \eqref{eq:mix-2ch-problem} may not be attained but only approached at the boundary $\partial \mathcal{K}_\beta$.
To capture this limiting behavior, we extend $J_{\mathrm{mix}}$ to $\partial \mathcal{K}_\beta$.
By Corollary~\ref{corollary:boundary-set}, any $K \in \partial \mathcal{K}_\beta$ satisfies $\|\mathbf{T}_{\infty}(K)\|_{\mathcal{H}_\infty} \!=\! \beta$, so Lemma~\ref{lem:BRL-X} no longer applies and \eqref{eq:mix-2ch-Riccati-XK} does not admit stabilizing solution.
Nonetheless, it is known that \eqref{eq:mix-2ch-Riccati-XK} still admits a unique \textit{minimal} solution $X_{K}$ with eigenvalues of $A_K+\beta^{-2}X_KS_K$ in the closed left half-plane \cite[Corollary 13.13, Lemma 13.17]{zhou1996robust}.
We thus define the boundary cost as
\begin{equation} \label{eq:cost-value-boundary}
\tilde{J}_{\mathrm{mix}}(K) = \Tr((Q_2+K^\tr R_2 K)X_K), \quad \forall K \in \partial\mathcal{K}_\beta,
\end{equation}
where $X_K$ is the \textit{minimal} solution to \eqref{eq:mix-2ch-Riccati-XK}. 

The above extension is continuous: as $K \!\in\! \mathcal{K}_\beta$ approaches a boundary point $K_0 \!\in\! \partial \mathcal{K}_\beta$, the cost converges to $\tilde{J}_{\mathrm{mix}}(K_0)$. 

\begin{lemma} \label{thm:continuity-mix-cost}
    Suppose Assumptions~\ref{assumption:1} and \ref{assumption:positive-definiteness} hold. Given any boundary point $K_0 \in \partial \mathcal{K}_\beta$, we have  
    $$
    \lim_{K\to{K}_0, K \in \mathcal{K}_\beta} {J}_{\mathrm{mix}}(K)=\tilde{J}_{\mathrm{mix}}(K_0).
    $$
\end{lemma}

The proof relies on a nontrivial continuity result for Riccati minimal solutions \cite[Theorem 11.2.1]{lancaster1995algebraic}, with details deferred to Appendix~\ref{appendix:proof-continuity}.
As a direct consequence of Lemma~\ref{thm:continuity-mix-cost}, the cost ${J}_{\mathrm{mix}}$ is non-coercive, remaining finite as $K$ with $\|K\| < \infty$ approaches the boundary of $\mathcal{K}_\beta$. 
Nevertheless, since $J_{\mathrm{mix}}$ serves as a pointwise upper bound on the coercive LQR cost, it exhibits \textit{partial} coercivity: $J_{\mathrm{mix}}(K) \to \infty$ as $\|K\| \to \infty$.

\begin{example}[Non-coercivity]\label{example:mix-cost-nonconvex-noncoercive}
    Consider the problem instance from Example~\ref{example:mix-set-nonconvexity-unbounded}. We set $\beta=3.5$ and $Q_2\!=\!R_2\!=\!I_2$ to define $J_{\mathrm{mix}}$. For comparison, we also consider the limiting case $\beta=\infty$, which reduces to the LQR cost $J_{\mathrm{LQR}}$. The global optima are $K^\ast=-0.4193 I_2$ for $J_{\mathrm{mix}}$ and $K^\ast_{\mathrm{LQR}}=(1-\sqrt{2})I_2$ for $J_\mathrm{LQR}$. 
    Fig.~\ref{fig:mix-cost-not-coercive} depicts $J_{\mathrm{mix}}$ for $K=K^\ast + \left[\begin{smallmatrix}
        0 & k_{12} \\ k_{21} & 0
    \end{smallmatrix}\right]$, while Fig.~\ref{fig:lqr-cost-coercive} plots $J_\mathrm{LQR}$ for $K=K^\ast_{\mathrm{LQR}} + \left[\begin{smallmatrix}         0 & k_{12} \\ k_{21} & 0
    \end{smallmatrix}\right]$.
    As shown in Fig.~\ref{fig:mix-cost-not-coercive}, the cost $J_{\mathrm{mix}}$ stays bounded as $K$ with $\|K\|<\infty$ approaches the boundary $\partial\mathcal{K}_\beta$, unlike $J_\mathrm{LQR}$ which diverges in Fig.~\ref{fig:lqr-cost-coercive}.
    \hfill $\qed$
\end{example} 

Theorem~\ref{theorem:analytical-functions} ensures that $J_{\mathrm{mix}}$ is infinitely differentiable. In particular, we have the following gradient formulas. 

\begin{lemma} \label{theorem:policy-gradient-mixed}
    Suppose Assumption~\ref{assumption:1} holds.
    For any $K \in \mathcal{K}_\beta$, the policy gradient of $J_{\mathrm{mix}}$ is given by
    \begin{subequations} \label{eq:gradient-J}
        \begin{equation}
            \nabla J_{\mathrm{mix}}(K)\!=\!2(R_2 K\!+\!B^\tr \Gamma_K+\beta^{-2}R_\infty K X_K \Gamma_K)X_K, 
        \end{equation}
        where ${X}_K$ is the stabilizing solution to the Riccati equation \eqref{eq:mix-2ch-Riccati-XK}, and $\Gamma_K$ is the unique solution to the Lyapunov equation 
        \begin{equation} \label{eq:mix-2ch-grad-Lyapunov-eq}
            \tilde{A}_K^\tr \Gamma_K+\Gamma_K\tilde{A}_K + Q_2+K^\tr R_2K=0
        \end{equation}
        with $\tilde{A}_K:= A_K+\beta^{-2}X_KS_K$ being stable. 
    \end{subequations}
\end{lemma} 

Lemma~\ref{theorem:policy-gradient-mixed} applies to the general two-channel case.
In the single-channel case where $z_2=z_\infty$, an alternative gradient expression can be derived using \eqref{eq:mix-1ch-cost-2}. 

\begin{lemma} \label{theorem:policy-gradient-mixed-signle-channel}
    Suppose Assumption~\ref{assumption:1} holds. If $z_2=z_\infty$ in \eqref{eq:transfer-function-Tz2w} and \eqref{eq:transfer-function-Tziw}, the policy gradient of $J_{\mathrm{mix}}$ can be evaluated as
    \begin{subequations} \label{eq:gradient-Jtilde}
        \begin{equation}
           \nabla J_{\mathrm{mix}}(K)=2(RK+B^\tr {P}_K)\Lambda_K, \quad \forall K \in \mathcal{K}_{\beta}
        \end{equation}
        where ${P}_K$ is the stabilizing solution to the Riccati equation \eqref{eq:mix-1ch-Riccati-PK}, and $\Lambda_K$ is the solution to the Lyapunov equation
        \begin{align} \label{eq:mix-1ch-grad-Lyapunov-eq}
            \hat{A}_{K}\Lambda_K+ \Lambda_K \hat{A}_{K}^\tr + W=0
        \end{align}
        with $\hat{A}_{K}:=A_K+\beta^{-2}W{P}_K$ being stable. 
    \end{subequations}
\end{lemma} 

The formula in \eqref{eq:gradient-J} also applies to the single-channel case by setting $Q_2\!=\!Q_\infty$ and $R_2\!=\!R_\infty$.
While mathematically equivalent, the alternative expression in \eqref{eq:gradient-Jtilde} is more useful for analysis.
Our proofs of Lemmas~\ref{theorem:policy-gradient-mixed} and \ref{theorem:policy-gradient-mixed-signle-channel} build on the strategy of \cite[Appendix B.4]{tang2021analysis}, adapting sensitivity analysis of Lyapunov equations to the Riccati settings. Details are provided in Appendices~\ref{app:proof-mix-grad-2ch}.

\begin{remark}[Comparison with LQR]
Lemmas~\ref{theorem:policy-gradient-mixed} and \ref{theorem:policy-gradient-mixed-signle-channel} show that computing $\nabla J_{\mathrm{mix}}$ requires solving one Riccati and one Lyapunov equation, slightly more involved than the LQR case, which only needs two Lyapunov equations \cite[Section IV]{levine1970determination}.
Note that as $\beta \!\to\! \infty$, $J_{\mathrm{mix}}$ in \eqref{eq:mix-cost-2ch} and \eqref{eq:mix-1ch-cost-2} reduces to $J_{\mathrm{LQR}}$, and both gradient formulas simplify to $\nabla J_{\mathrm{LQR}}$.
In particular, we have $\nabla J_{\mathrm{LQR}}(K) \!=\! 2(R_2 K+B^\tr \Gamma_K)X_K$, where ${X}_K$ and $\Gamma_K$ are the solutions to the Lyapunov equations 
\begin{equation} \label{eq:mix-2ch-LQR}
    \begin{aligned}
        &A_KX_K + X_KA_K^\tr + W =0,\\
        &A_K^\tr \Gamma_K+\Gamma_KA_K + Q_2+K^\tr R_2K=0.
    \end{aligned}
\end{equation} 
This is expected as \eqref{eq:mix-2ch-problem} reduces to LQR as $\beta \to \infty$. \hfill $\qed$ 
\end{remark}

\section{No Spurious Stationary Points}\label{section:no-spurious-stationary}

In this section, we characterize the global optimality of the mixed $\mathcal{H}_2/\mathcal{H}_\infty$ design \eqref{eq:mix-2ch-problem}. 

\subsection{Any stationary point is globally optimal} 

Since $\mathcal{K}_\beta$ is open, any local minimizer of $J_{\mathrm{mix}}$ must lie in its interior and thus must be a stationary point. Despite nonconvexity, we establish a key result that every stationary point (when it exists) of \eqref{eq:mix-2ch-problem} is globally optimal. 

\begin{theorem}
\label{thm:mix-stationary-global-optimality}
    Suppose Assumptions~\ref{assumption:1} and \ref{assumption:positive-definiteness} hold. For any policy $K\in\mathcal{K}_{\beta}$, if $K$ is a stationary point, i.e., $\nabla J_{\mathrm{mix}}(K) = 0$, then $K$ is a global minimizer of \eqref{eq:mix-2ch-problem}. 
\end{theorem}

Owing to the absence of spurious stationary points, the mixed $\mathcal{H}_2/\mathcal{H}_\infty$ design \eqref{eq:mix-2ch-problem} exhibits hidden convexity: \textit{every local minimum is globally optimal}. This parallels recent results on benign nonconvexity in state-feedback problems such as LQR \cite{fazel2018global,mohammadi2021convergence,watanabe2025revisiting} and $\mathcal{H}_\infty$ control \cite{guo2022global,zheng2024benign,tang2023global}. 

Theorem~\ref{thm:mix-stationary-global-optimality} is most closely related to \cite[Proposition A.5]{zhang2021policy}, which establishes the global optimality of all stationary points in the single-channel case. Its proof relies on a game-theoretic formulation, whose extension to the two-channel case \eqref{eq:mix-2ch-problem} remains unclear.
In contrast, our approach builds on convex reformulations via the recently developed \ECL{} framework \cite{zheng2024benign}, which provides greater transparency and directly deals with the general two-channel problem.
In this sense, \cite[Proposition A.5]{zhang2021policy} appears as a special case of Theorem~\ref{thm:mix-stationary-global-optimality}.
Our proof involves an explicit \ECL{} construction, which requires careful analysis of (non)strict Riccati inequalities.
We provide the details in Section~\ref{section:ECL}.

Building on Lemma~\ref{theorem:policy-gradient-mixed} and Theorem~\ref{thm:mix-stationary-global-optimality}, we further derive stationarity conditions that characterize global optima of problem \eqref{eq:mix-2ch-problem}. To simplify the expression, we assume $R_\infty=\alpha^2R_2$ with $\alpha\geq0$ as a design parameter \cite{bernstein1989lqg}.

\begin{corollary}
\label{thm:mixed-2ch-optimality-conditions}
    Suppose Assumptions~\ref{assumption:1} and \ref{assumption:positive-definiteness} hold, and assume $R_\infty=\alpha^2R_2$ for some $\alpha\geq0$. Then a policy $K\in\mathcal{K}_\beta$ is global optimal for \eqref{eq:mix-2ch-problem} if and only if there exist matrices $\Gamma\succeq0$ and $X\succ0$ satisfying the following conditions:
    \begin{subequations}
    \label{eq:mix-2ch-optimality-conditions}
    \begin{align}
        &K=-R_2^{-1}B^\tr \Gamma\left(I+\beta^{-2}\alpha^2 X\Gamma\right)^{-1}, \label{eq:mix-2ch-optimality-conditions-mixed-a}\\
        &A_KX + XA_K^\tr + \beta^{-2}XS_K X + W = 0, \label{eq:mix-2ch-optimality-conditions-mixed-b} \\
        &\tilde{A}_K^\tr \Gamma+\Gamma \tilde{A}_K + Q_2+K^\tr R_2K=0, \label{eq:mix-2ch-optimality-conditions-mixed-c}
    \end{align}
    \end{subequations}
    {where $\tilde{A}_K \!:=\! A_K+\beta^{-2}XS_K$ is further Hurwitz.}
\end{corollary}

\begin{proof}
    We prove the equivalence in two directions.

    $\Rightarrow$. Suppose $K\in\mathcal{K}_\beta$ is globally optimal for \eqref{eq:mix-2ch-problem}. Since $\mathcal{K}_\beta$ is open, we must have $\nabla J_{\mathrm{mix}}(K) = 0$. Using the gradient formula \eqref{eq:gradient-J} and setting $\nabla J_{\mathrm{mix}}(K) = 0$ yields
    \begin{equation} \label{eq:gradient-zero-2-channel}
        2(R_2 K+B^\tr \Gamma_K+\beta^{-2}R_\infty K X_K \Gamma_K)X_K = 0,
    \end{equation}
    where $X_K$ is the unique stabilizing solution to the Riccati equation \eqref{eq:mix-2ch-Riccati-XK} and $\Gamma_K$ solves the Lyapunov equation \eqref{eq:mix-2ch-grad-Lyapunov-eq}.
    Since $W \succ 0$, the bounded real lemma ensures that $X_K \succ 0$. In addition, \eqref{eq:mix-2ch-grad-Lyapunov-eq} implies that $\Gamma_K\succeq0$. Thus, the matrix $I+\beta^{-2}\alpha^2 X_K\Gamma_K$ is invertible.
    Let $\Gamma:=\Gamma_K$ and $X:=X_K$. Then using $R_\infty=\alpha^2R_2$, we can rearrange \eqref{eq:gradient-zero-2-channel} to obtain \eqref{eq:mix-2ch-optimality-conditions-mixed-a}. It is clear that \eqref{eq:mix-2ch-Riccati-XK} and \eqref{eq:mix-2ch-grad-Lyapunov-eq} corresponds to \eqref{eq:mix-2ch-optimality-conditions-mixed-b} and \eqref{eq:mix-2ch-optimality-conditions-mixed-c}, respectively. Thus, all three optimality conditions are satisfied.
    
    $\Leftarrow$.  
    Suppose that there exist $X\succ0$ and $\Gamma\succeq0$ satisfying \eqref{eq:mix-2ch-optimality-conditions-mixed-a} to \eqref{eq:mix-2ch-optimality-conditions-mixed-c}. From \eqref{eq:mix-2ch-optimality-conditions-mixed-b} and the stability of $\tilde{A}_K$, Lemma~\ref{lem:BRL-X} implies that $K\in\mathcal{K}_\beta$.
    Substituting the expression for $K$ in \eqref{eq:mix-2ch-optimality-conditions-mixed-a} into the gradient formula \eqref{eq:gradient-J} confirms that $\nabla J_{\mathrm{mix}}(K)=0$. We then conclude by Theorem~\ref{thm:mix-stationary-global-optimality} that $K$ is a global minimizer of \eqref{eq:mix-2ch-problem}.  
\end{proof}

Corollary~\ref{thm:mixed-2ch-optimality-conditions} provides necessary and sufficient conditions for a policy $K \in \mathcal{K}_\beta$ to be globally optimal. Similar conditions were derived in \cite[Theorem 4.1]{bernstein1989lqg} via Lagrange multipliers, but were only shown to be necessary.

\begin{figure*}[t]
    \centering
    \subfigure[Two-channel: $\beta=1$]{%
        \includegraphics[width=0.25\textwidth]{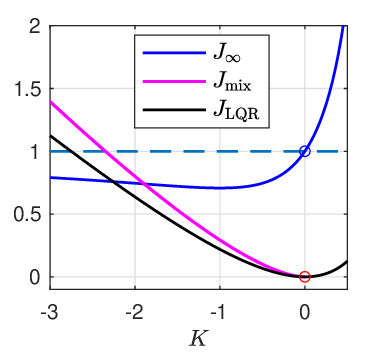}
        \label{fig:1dim-beta1}
    }\hspace{4mm}
    \subfigure[Two-channel: $\beta=2$]{%
        \includegraphics[width=0.25\textwidth]{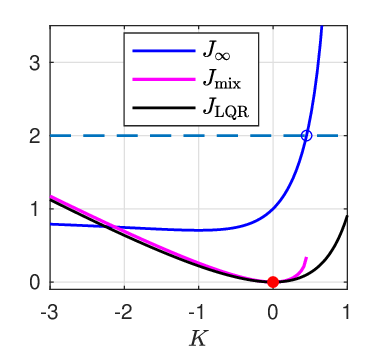}
        \label{fig:1dim-beta2}
    }\hspace{4mm}
    \subfigure[Single-channel: $\beta=1$]{%
        \includegraphics[width=0.25\textwidth]{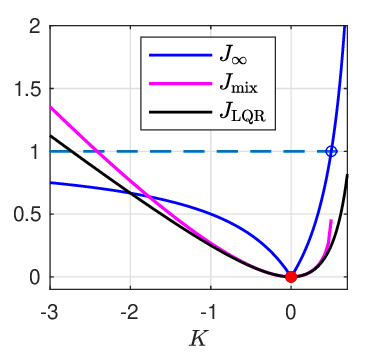}
        \label{fig:1dim-beta2-single}
    }
    \caption{Optimization landscapes of $J_{\mathrm{mix}}$, $J_\infty$, and $J_{\mathrm{LQR}}$ in Example~\ref{ex:1dim-two-channel}. (a) Infimum $J_1^\ast$ not attained (hollow red circle marks the boundary). (b) Minimum $J_2^\ast$ attained (solid red dot). (c) Single-channel case: minimum attained (solid red dot).}
    \label{fig:simple-example-mixed-1dim}
\end{figure*}

Clearly, Corollary~\ref{thm:mixed-2ch-optimality-conditions} also applies to the single-channel case. However, the alternative formulation in \eqref{eq:mix-1ch-cost-2} leads to a simpler gradient expression in \eqref{eq:gradient-Jtilde}, further offering additional insight into global optimality. 
To see this, with $W\succ 0$, the solution $\Lambda_K$ to the Lyapunov equation \eqref{eq:mix-1ch-grad-Lyapunov-eq} is positive definite, and therefore, the stationarity condition $\nabla{J_{\mathrm{mix}}(K)} = 0$ implies
$
K = -R^{-1}B^\tr {P}_K. 
$
Substituting this into the Riccati equation \eqref{eq:mix-1ch-Riccati-PK} results in a single Riccati equation, as summarized below. 

\begin{corollary}
\label{thm:mixed-1ch-optimality-conditions}
    Under Assumptions~\ref{assumption:1} and \ref{assumption:positive-definiteness}, if $z_2=z_\infty$ in \eqref{eq:transfer-function-Tz2w} and \eqref{eq:transfer-function-Tziw}, a policy $K\in\mathcal{K}_\beta$ is globally optimal for \eqref{eq:mix-2ch-problem} if and only if there exists a matrix {${P}\succeq0$} such that: 
    \begin{subequations}
    \label{eq:mix-1ch-optimality-conditions}
    \begin{align}
        &K=-R^{-1}B^\tr {P}, \label{eq:mix-1ch-optimality-conditions-a}\\
        &A^\tr {P}+{P}A+{P} (\beta^{-2}W-BR^{-1}B^\tr) {P}+Q=0, \label{eq:mix-1ch-optimality-conditions-b}
    \end{align}
    \end{subequations}
    where $A+(\beta^{-2}W-BR^{-1}B^\tr)P$ is further Hurwitz.
\end{corollary}

The proof is similar to that of Corollary~\ref{thm:mixed-2ch-optimality-conditions}. We give some details in Appendix~\ref{appendix:proof-optimality-signle-channel}.
Compared with the coupled conditions in \eqref{eq:mix-2ch-optimality-conditions}, the conditions in \eqref{eq:mix-1ch-optimality-conditions} are much simpler, requiring solving only a single Riccati equation independent of $K$.

\begin{remark}[Connection to LQR] \label{remark:LQR-case}
As $\beta \!\to\! \infty$, both optimality conditions in the single- and two-channel cases, \eqref{eq:mix-2ch-optimality-conditions} and \eqref{eq:mix-1ch-optimality-conditions}, reduce to the classical LQR Riccati equation
\begin{equation*} \label{eq:LQR-Riccati}
    A^\tr {P}+{P}A-{P}BR^{-1}B^\tr {P}+Q=0,
\end{equation*}
with the optimal policy $K=-R^{-1}B^\tr {P}$, where ${P}$ is the stabilizing solution. \hfill $\qed$ 
\end{remark}

\begin{remark}[Connection to $\mathcal{H}_\infty$ suboptimal control]\label{remark:connection-to-Hinf-control}
The Riccati equation \eqref{eq:mix-1ch-optimality-conditions-b} is closely tied to state-feedback $\mathcal{H}_\infty$ suboptimal control. 
By \cite[Theorem 17.6]{zhou1996robust} and \cite[Theorem 20.2.1]{lancaster1995algebraic}, it admits a unique stabilizing solution $P$ if and only if there exists a policy $K\!\in\!\mathcal{K}_\beta$. 
One suboptimal policy is given by $ K\!=\!-R^{-1}B^\tr {P}$. 
As shown in Corollary~\ref{thm:mixed-1ch-optimality-conditions}, this $K$ also solves problem \eqref{eq:mix-2ch-problem} with $z_2=z_\infty$.
This highlights the close connection of \eqref{eq:mix-2ch-problem} to several classical formulations, including maximum entropy $\mathcal{H_\infty}$ control \cite{mustafa1990minimum}, risk-sensitive control \cite{jacobson1973optimal}, and zero-sum dynamic games \cite{bacsar2008h,watanabe2025semidefinite}.
\hfill $\qed$  
\end{remark}

\subsection{Existence of stationary points}\label{subsubsec:existence-stationary}

Theorem~\ref{thm:mix-stationary-global-optimality} and  Corollaries~\ref{thm:mixed-2ch-optimality-conditions}-\ref{thm:mixed-1ch-optimality-conditions} do not ensure the existence of stationary points.
In particular, the optimality conditions \eqref{eq:mix-2ch-optimality-conditions} may be infeasible.
We summarize this fact below. 

\begin{fact} \label{fact:solvability-2ch-mix}
    Under Assumptions~\ref{assumption:1} and \ref{assumption:positive-definiteness}, the mixed $\mathcal{H}_2/\mathcal{H}_\infty$ control \eqref{eq:mix-2ch-problem} may admit no stationary points in $\mathcal{K}_\beta$, and its optimal value may not be attained within $\mathcal{K}_\beta$. 
\end{fact}

In fact, the infeasibility of \eqref{eq:mix-2ch-optimality-conditions} stems from an overly stringent robustness requirement. Relaxing the constraint by increasing $\beta$ ensures the existence of a stationary point.

\begin{theorem} \label{prop:existence-2ch-large-beta}
Under Assumption~\ref{assumption:1}, there exists a threshold $\bar{\beta} > 0$ such that for all $\beta > \bar{\beta}$, problem \eqref{eq:mix-2ch-problem} admits a stationary point $K\in\mathcal{K}_\beta$. Equivalently, there exists $(K,X,\Gamma)$ satisfying \eqref{eq:mix-2ch-optimality-conditions} where $X$ is the stabilizing solution to \eqref{eq:mix-2ch-optimality-conditions-mixed-b}.
\end{theorem}

Note that this existence result does not rely on the assumption $R_\infty=\alpha^2 R_2$. The proof uses a perturbation argument around the LQR case. 
The key insight is that as $\beta \to \infty$, the optimality conditions \eqref{eq:mix-2ch-optimality-conditions} converge to those of LQR (Remark~\ref{remark:LQR-case}), where a stationary point always exists. The details are technically involved, and we present them in Appendix~\ref{appendix:proof-2ch-existence-stationary}.
    
In contrast, the single-channel case always admits a stationary point attaining the optimal value, thanks to its connection to $\mathcal{H}_\infty$ suboptimal control (Remark~\ref{remark:connection-to-Hinf-control}). 

\begin{theorem} \label{theorem:unique-stationary-point}
    Under Assumptions~\ref{assumption:1} and \ref{assumption:positive-definiteness}, the single-channel problem \eqref{eq:mix-2ch-problem} with $z_2 \!=\! z_\infty$ admits a unique stationary point $K\in\mathcal{K}_\beta$, given by $K=-R^{-1}B^\tr {P}$, where $P$ is the stabilizing solution to the Riccati equation \eqref{eq:mix-1ch-optimality-conditions-b}.
\end{theorem}

\begin{proof} 
    By \cite[Theorem 17.6]{zhou1996robust} and \cite[Theorem 20.2.1]{lancaster1995algebraic}, \eqref{eq:mix-1ch-optimality-conditions-b} admits a stabilizing solution if and only if $\mathcal{K}_\beta$ is nonempty, satisfied by Assumption~\ref{assumption:1}.
    Then, under Assumption~\ref{assumption:positive-definiteness}, the stationary point is given by \eqref{eq:mix-1ch-optimality-conditions-a}, and its uniqueness follows from that of the stabilizing solution.
\end{proof}

We provide a simple example illustrating Theorems~\ref{prop:existence-2ch-large-beta} and~\ref{theorem:unique-stationary-point}, showing that the existence of stationary points depends on $\beta$.  

\begin{example} \label{ex:1dim-two-channel}
    Consider a scalar instance of \eqref{eq:mix-2ch-problem} with $A=-1$, $B=B_w=1$, $Q_2=0$, $Q_\infty=1$, $R_2=R_\infty=1$.
    The stabilizing set is $\mathcal{K}\!=\!\{k\in\mathbb{R}: k\!<\!1\}$.
    The $\calH_\infty$ norm can be computed in closed form as
    $J_\infty(k):=\|\mathbf{T}_{\infty}(k)\|_{\mathcal{H}_\infty}=\sqrt{1+k^2}/(1-k)$ for all $k<1$, which yields the feasible set 
    \[
    \mathcal{K}_\beta \!=\! \{k: 1+k^2\!<\!\beta^2(1-k)^2,\ k\!<\!1\}.
    \]
    With $J_{\mathrm{mix}}(k) \!=\! k^2 X_k$ and $X_k$ the stabilizing solution to $\beta^{-2}(k^2+1)X_k^2+2(k-1)X_k+1=0$, problem \eqref{eq:mix-2ch-problem} reads
    \begin{align*} \label{eq:mixed-ex-1dim-beta2}
        J_\beta^\ast \!:=\! \inf_{k\in\mathcal{K}_\beta} k^2 X_k,\ \ \scalebox{1}{$
    X_k=\frac{1-k-\sqrt{(k-1)^2-\beta^{-2}(k^2+1)}}{\beta^{-2}(k^2+1)}
    $.}
    \end{align*}
    
    \textbf{Case 1: $\beta=1$.} 
    We have $\mathcal{K}_1=\{k:k<0\}$.
    The infimum $J_1^\ast=0$ is not attained in $\mathcal{K}_1$ but approached as $k \to 0^-$ and $X_k \to 1^+$ (see Figure~\ref{fig:1dim-beta1}).
    
    \textbf{Case 2: $\beta=2$.} 
    We have $\mathcal{K}_2\!=\!\{k:k\!<\!(4\!-\!\sqrt{7})/3\}$. The infimum $J_2^\ast=0$ is attained by the optimal policy $k^\ast=0$ (see Figure~\ref{fig:1dim-beta2}). One can verify by Lemma~\ref{theorem:policy-gradient-mixed} $\nabla J_{\mathrm{mix}}(k^\ast)=0$.  

    \textbf{Case 3 (single-channel): $\beta=1$.} 
    We consider $Q\!=\!0$ and $R\!=\!1$, which preserves the LQR cost but alters the $\calH_\infty$
    constraint (blue curve in Figure~\ref{fig:1dim-beta2-single}).
    Here, $J_{\mathrm{mix}}$ always admits a stationary point, showing that problem structure beyond $\beta$ also affects the existence of optimal policies.
\end{example}

\section{Analysis via Extended Convex Lifting} \label{section:ECL}

This section exploits the recent \ECL{} framework to prove the global optimality result in Theorem~\ref{thm:mix-stationary-global-optimality}, details the \ECL{} construction for problem \eqref{eq:mix-2ch-problem}, and discusses the solvability of the associated convex reformulation.

\subsection{A framework of Extended Convex Lifting}

The \ECL{} framework \cite{zheng2024benign} studies a class of nonconvex optimization problems using convex tools. The key idea is simple and based on a change of variables. Let $f:\mathbb{R}^n \to \mathbb{R}$ be smooth and nonconvex. Suppose there exists a smooth bijection $y = \phi(x)$ such that $g(y):= f(\phi^{-1}(y))$ becomes convex. It is then clear that $ \min f(x) = \min g(y)$ and every stationary point of $f$ is globally optimal. Here, the mapping $y = \phi(x)$ acts as direct convexification.

The \ECL{} generalizes this idea to a broader class of constrained nonconvex problems that may not admit direct convexification.
Indeed, such problems frequently arise in control, including \eqref{eq:mix-2ch-problem}. 
We first review the definitions and key results of \ECL{}, presenting a simplified version tailored for clarity and relevance to our setting.

For a function $f: \mathcal{D} \to \mathbb{R}$ with an open domain $D \subseteq \mathbb{R}^d$, we define its strict and non-strict epigraphs by
\begin{align*}
    &\operatorname{epi}_>(f) \coloneqq \{(x,\gamma)\in\mathcal{D}\times\mathbb{R}\mid
    \gamma>f(x)
    \},\\ 
    &\operatorname{epi}_\geq(f) \coloneqq \{(x,\gamma)\in\mathcal{D}\times\mathbb{R}\mid
    \gamma\geq f(x)
    \}. 
\end{align*}
\begin{definition}[Extended Convex Lifting \cite{zheng2024benign}] \label{definition:ECL}
    Let $f:\mathcal{D}\rightarrow \mathbb{R}$ be continuous, where $\mathcal{D}\subseteq\mathbb{R}^{d}$ is open. We say that a tuple $(\mathcal{L}_{\mathrm{lft}},\mathcal{F}_{\mathrm{cvx}},\Phi)$ is an \ECL{} of $f$ if the following hold:
    \begin{enumerate}
        \item $\mathcal{L}_{\mathrm{lft}} \subseteq \mathbb{R}^{d}\times\mathbb{R}\times\mathbb{R}^{d_\xi}$ is a lifted set with an extra variable $\xi \in \mathbb{R}^{d_\xi}$, such that its canonical projection onto the first $d\!+\!1$ coordinates, given by
        $ 
            \pi_{x,\gamma}(\mathcal{L}_{\mathrm{lft}})
        \!=\! \{(x,\gamma):\exists \xi \in\mathbb{R}^{d_{\xi}} \; \text{s.t.}\; (x,\gamma,\xi)\! \in \! \mathcal{L}_{\mathrm{lft}}\},
        $ 
        satisfies
        \begin{equation} \label{eq:epi-graph-lifting}
            \operatorname{epi}_>(f)
            \subseteq \pi_{x,\gamma}(\mathcal{L}_{\mathrm{lft}})\subseteq 
            \operatorname{cl}\operatorname{epi}_{\geq} (f).
        \end{equation}
        
        \item $\mathcal{F}_{\mathrm{cvx}} \subseteq \mathbb{R}\times \mathbb{R}^{d+d_\epsilon} $ is a convex set and $\Phi$ is a $C^2$ diffeomorphism from $\mathcal{L}_{\mathrm{lft}}$ to $\mathcal{F}_{\mathrm{cvx}}$.
        
        \item For any $(x,\gamma,\xi)\in \mathcal{L}_{\mathrm{lft}}$, we have $\Phi(x,\gamma,\xi) = (\gamma, \zeta_1) \in \mathcal{F}_{\mathrm{cvx}}$ for some $\zeta_1\in\mathbb{R}^{d+d_\epsilon}$. In other words, the mapping $\Phi$ directly outputs $\gamma$ in the first component.
    \end{enumerate}
\end{definition}

This procedure generalizes direct convexification by introducing epigraphs and a lifting variable $\xi$.
The set inclusion \eqref{eq:epi-graph-lifting} also adds flexibility.
Similar to direct convexification, the existence of an \ECL{} enables the minimization of $f$ over $\mathcal{D}$ to be reformulated as a convex problem.
Specifically, we have the following equivalence \cite{zheng2024benign}
\begin{equation}\label{eq:ECL_equivalent_opt}
\inf_{x\in\mathcal{D}} f(x) = \inf_{(\gamma,\zeta_1)\in\mathcal{F}_{\mathrm{cvx}}} \gamma,
\end{equation}
where the right-hand side is convex by the \ECL{} construction. 

We here define \textit{non-degenerate} points of $f$.
 
\begin{definition}[Non-degenerate points \cite{zheng2024benign}] \label{definition:degenerate-points}
Given an \ECL{} $(\mathcal{L}_{\mathrm{lft}},\mathcal{F}_{\mathrm{cvx}},\Phi)$ for $f$, a point $x\in\mathcal{D}$ is called \textit{non-degenerate} if
$
(x,f(x))\in\pi_{x,\gamma}(\mathcal{L}_{\mathrm{lft}}).
$
\end{definition}

Lastly, the existence of an \ECL{} certifies the global optimality of \textit{non-degenerate} stationary points of $f$. 

\begin{lemma}[{\cite[Theorem 3.2]{zheng2024benign}}] \label{theorem:ECL-Guarantees}
    Let $f:\mathcal{D}\rightarrow\mathbb{R}$ be differentiable, where $\mathcal{D}\subseteq\mathbb{R}^d$ is open. Suppose $f$ admits an \ECL{} $(\mathcal{L}_{\mathrm{lft}},\mathcal{F}_{\mathrm{cvx}},\Phi)$. Then, any non-degenerate stationary point $x^\ast$ satisfying $\nabla f(x^\ast)\!=\!0$ is a global minimizer of $f$ over $\mathcal{D}$. 
\end{lemma}

\subsection{Proof of Theorem~\ref{thm:mix-stationary-global-optimality} via \ECL{} construction}

In this section, we construct an \ECL{} for problem \eqref{eq:mix-2ch-problem}, and show that all feasible policies in $\mathcal{K}_\beta$ are non-degenerate under Assumptions~\ref{assumption:1} and \ref{assumption:positive-definiteness}. Theorem~\ref{thm:mix-stationary-global-optimality} then follows directly from Theorem~\ref{theorem:ECL-Guarantees}.
The \ECL{} construction details are involved, which rely on the (non)-strict bounded real lemma.

We start by characterizing the non-strict epigraph of the extended cost $\tilde{J}_{\mathrm{mix}}$, which is central to the \ECL{} construction.
Working with the extended cost simplifies the treatment of boundary cases and also facilitates subsequent proofs.

\begin{lemma} \label{lem:inequalities-for-J}
Consider the extended cost $\tilde{J}_{\mathrm{mix}}:\mathrm{cl}(\mathcal{K}_\beta)\to\mathbb{R}$ defined in \eqref{eq:cost-value-boundary}. Under Assumptions~\ref{assumption:1} and \ref{assumption:positive-definiteness}, a pair $(K,\gamma)\!\in\!\operatorname{epi}_{\geq}(\tilde{J}_{\mathrm{mix}})$ if and only if there exists some $X\!\succ\! 0$ such that
\begin{align*}
    &\gamma \geq \Tr(Q_2X)+\Tr(R_2 KXK^\tr),\\
    &A_KX + XA_K^\tr + W+ \beta^{-2}XS_K X \preceq 0.
\end{align*}
    
\end{lemma}

The proof of Lemma~\ref{lem:inequalities-for-J} is deferred to Appendix~\ref{subsec:proof-strict-lmis}. We now construct an \ECL{} for \eqref{eq:mix-2ch-problem} in three steps.

\begin{subequations} \label{eq:ECL-2ch-all}
\vspace{3pt}
\noindent \textbf{Step 1: Lifted set.} Motivated by Lemma~\ref{lem:inequalities-for-J}, we introduce a lifted set $\mathcal{L}_\mathrm{lft}$ with an extra variable $X$
\begin{equation} \label{eq:lifted-set-mixed-state-feedback}
    \scalebox{0.85}{$
    \mathcal{L}_{\mathrm{lft}}
    \coloneqq \left\{(K,\gamma,X)\left|  \begin{aligned}
    &K \in \mathbb{R}^{m \times n},\ \gamma\in\mathbb{R},\ X\succ 0,\\
    &\gamma\geq \Tr(Q_2X) + \Tr(R_2 KXK^\tr), \\
    &A_KX \!+\! XA_K^\tr \!+\! \beta^{-2}XS_K X \!+\! W \preceq 0
    \end{aligned} \right.
    \right\}.
    $}
\end{equation}

\vspace{3pt}
\noindent \textbf{Step 2: Convex set.} We define a convex set 
\begin{equation} \label{eq:convex-set-mixed-state-feedback}
    \scalebox{0.85}{$
    \mathcal{F}_{\mathrm{cvx}} \coloneqq
    \left\{
    (\gamma, X, Y)\left| 
    \begin{aligned}
        \gamma\!\in\!\mathbb{R},\ X\!\succ\!0,\ Y\!\in\!\mathbb{R}^{m\times n},\ \mathbb{F}_{\mathrm{cvx}}\!\preceq\!0,\\
            \gamma\geq\Tr(Q_2X)+\Tr(R_2YX^{-1}Y^\tr)\\ 
    \end{aligned}\right.
    \right\},
    $}
\end{equation}
where $\mathbb{F}_{\mathrm{cvx}}$ is defined in \eqref{eq:mix-feasible-set-connectivity}.

\vspace{3pt}
\noindent \textbf{Step 3: Diffeomorphism.} Using the classical change of variables $Y\!=\!KX$, we define the mapping
\begin{equation} \label{eq:mapping-mixed-state-feedback}
\Phi(K,\gamma, X) \coloneqq (\gamma, X,KX),\quad\forall (K,\gamma,X)\in\mathcal{L}_{\mathrm{lft}}.
\end{equation}
    
\end{subequations}

Lemma~\ref{lem:inequalities-for-J} immediately yields the following key result, which validates the \ECL{} construction.

\begin{corollary}  \label{coro:epi-pi-same}
Under Assumptions~\ref{assumption:1} and \ref{assumption:positive-definiteness}, we have
\[
\operatorname{epi}_\geq(\tilde{J}_{\mathrm{mix}})\!=\!\pi_{K,\gamma}(\mathcal{L}_{\mathrm{lft}}).
\]
\end{corollary}

The following result validates that the constructed tuple $(\mathcal{L}_{\mathrm{lft}}, \mathcal{F}_{\mathrm{cvx}},\Phi)$ forms an \ECL{} for \eqref{eq:mix-2ch-problem}.

\begin{theorem} \label{proposition:mixed-lifting}
    Under Assumptions~\ref{assumption:1} and \ref{assumption:positive-definiteness}, the canonical projection of $\mathcal{L}_{\mathrm{lft}}$ onto $(K,\gamma)$, denoted $\pi_{K,\gamma}(\mathcal{L}_{\mathrm{lft}})$, satisfies
    \begin{align} \label{eq:ECL-inclusion}
        \operatorname{epi}_\geq(J_{\mathrm{mix}}) \subseteq \pi_{K,\gamma}(\mathcal{L}_{\mathrm{lft}})
        = \operatorname{cl}\operatorname{epi}_\geq(J_{\mathrm{mix}}).
    \end{align}
    Moreover, the mapping $\Phi$ given by \eqref{eq:mapping-mixed-state-feedback} is a $C^2$ (and in fact $C^\infty$) diffeomorphism from $\mathcal{L}_{\mathrm{lft}}$ in \eqref{eq:lifted-set-mixed-state-feedback} to $\mathcal{F}_{\mathrm{cvx}}$ in \eqref{eq:convex-set-mixed-state-feedback}. 
\end{theorem}

Note that the set inclusion in \eqref{eq:ECL-inclusion} follows directly from Corollary~\ref{coro:epi-pi-same}, whereas proving the set identity in \eqref{eq:ECL-inclusion} requires an additional argument showing that $\operatorname{epi}_\geq(\tilde{J}_{\mathrm{mix}}) = \operatorname{cl}\operatorname{epi}_\geq(J_{\mathrm{mix}})$, with proof details given in Appendix~\ref{subsection:proof-mixed-lifting}.

Since $\operatorname{epi}_>(J_{\mathrm{mix}}) \subseteq \operatorname{epi}_\geq(J_{\mathrm{mix}})$, the inclusion in \eqref{eq:ECL-inclusion} is stronger than that required by the \ECL{} definition in \eqref{eq:epi-graph-lifting}, which enlarges the set of non-degenerate points and allows \ECL{} to certify global optimality over $\mathcal{K}_\beta$.

\begin{remark}[Another valid \ECL{}] \label{remark:strict-ECL}
     In constructing $\mathcal{L}_{\mathrm{lft}}$ in \eqref{eq:lifted-set-mixed-state-feedback}, the nonstrict Riccati inequality is crucial for establishing $\operatorname{epi}_\geq(J_{\mathrm{mix}}) \subseteq \pi_{K,\gamma}(\mathcal{L}_{\mathrm{lft}})$.
     If instead the strict Riccati inequality is used, the construction still yields a valid \ECL{}, but only ensures the weaker inclusion $\operatorname{epi}_>(J_{\mathrm{mix}}) \subseteq \pi_{K,\gamma}(\mathcal{L}_{\mathrm{lft}})$.
\end{remark}

By Lemma~\ref{theorem:ECL-Guarantees}, an \ECL{} certifies the global optimality of any non-degenerate stationary points.
According to Definition~\ref{definition:degenerate-points} and the inclusion $\operatorname{epi}_\geq(J_{\mathrm{mix}}) \subseteq \pi_{K,\gamma}(\mathcal{L}_{\mathrm{lft}})$ shown in Theorem~\ref{proposition:mixed-lifting}, we have the following desirable property.
\begin{corollary}\label{fact:Hmix-non-degenearte}
    Using the \ECL{} $(\mathcal{L}_{\mathrm{lft}}, \mathcal{F}_{\mathrm{cvx}},\Phi)$ in \eqref{eq:ECL-2ch-all}, any $K \in \mathcal{K}_{\beta}$ is non-degenerate.
\end{corollary}

Considering Fact~\ref{fact:Hmix-non-degenearte} and the differentiability of $J_{\mathrm{mix}}$ over $\mathcal{K}_{\beta}$ (Lemma~\ref{theorem:policy-gradient-mixed}), the global optimality established in Theorem~\ref{thm:mix-stationary-global-optimality} follows directly from the \ECL{} guarantee in Lemma~\ref{theorem:ECL-Guarantees}.

\subsection{Convex reformulation and its solvability}

The \ECL{} construction also leads to a convex reformulation in \eqref{eq:ECL_equivalent_opt}, which offers further insight into problem \eqref{eq:mix-2ch-problem}.
\begin{theorem} \label{thm:mix-convex-reformulation}
    Consider problem \eqref{eq:mix-2ch-problem} and the convex set $\mathcal{F}_{\mathrm{cvx}}$ in \eqref{eq:convex-set-mixed-state-feedback}. Under Assumptions~\ref{assumption:1} and \ref{assumption:positive-definiteness}, we have
    \begin{equation} \label{eq:mix-2ch-LMI}
        \inf_{K \in \mathcal{K}_{\beta}} J_{\mathrm{mix}}(K) = \min_{(\gamma, X, Y)\in {\mathcal{F}}_{\mathrm{cvx}}} \gamma.
    \end{equation}
\end{theorem}

We use ``min'' in \eqref{eq:mix-2ch-LMI} to indicate that the optimum is always attained (see Proposition~\ref{thm:solvability}).
Unlike \eqref{eq:mix-2ch-problem} optimizing $K$ over the policy space $\mathcal{K}_\beta$, the convex reformulation in \eqref{eq:mix-2ch-LMI} optimizes $(\gamma,X,Y)$ over the higher-dimensional set ${\mathcal{F}}_{\mathrm{cvx}}$. 

The result most relevant to Theorem~\ref{thm:mix-convex-reformulation} is \cite[Theorem 4.3]{khargonekar1991mixed}, which gives a similar reformulation via strict LMIs (equations (30)-(31) therein).
While the strict formulation yields the same optimal value, it does not capture the global optimality of all stationary points (Theorem~\ref{thm:mix-stationary-global-optimality}).

We now address the solvability of the reformulation.

\begin{proposition} \label{thm:solvability}
    Let $\gamma^\ast$ be the (finite) optimal value of \eqref{eq:mix-2ch-problem}. Under Assumptions~\ref{assumption:1} and \ref{assumption:positive-definiteness}, the following statements hold.
    \begin{enumerate}
        \item The reformulation in \eqref{eq:mix-2ch-LMI}, i.e. {$\min_{(\gamma, X, Y)\in{\mathcal{F}}_{\mathrm{cvx}}} \gamma$}, admits a solution. That is, there exists $(\gamma^\ast,X^\ast, Y^\ast)\in{\mathcal{F}}_{\mathrm{cvx}}$ attaining the minimum $\gamma^\ast$.
        
        \item The corresponding policy $K^\ast\!:=\!Y^\ast(X^\ast)^{-1}$ lies in $\mathrm{cl}(\mathcal{K}_\beta)$ with its extended cost value $\tilde{J}_{\mathrm{mix}}(K^\ast)=\gamma^\ast$.
    \end{enumerate}
\end{proposition}

\begin{proof}
By Theorem~\ref{thm:mix-convex-reformulation}, $\gamma^\ast=\min_{(\gamma, X, Y)\in {\mathcal{F}}_{\mathrm{cvx}}} \gamma$.
We first note that $\gamma^\ast$ cannot be attained as $\|K\|\to\infty$ since the corresponding cost diverges. 
Now consider the case where $\gamma^\ast$ is attained in $\mathcal{K}_\beta$, i.e., there exists $K^\ast\in\mathcal{K}_\beta$ with $J_{\mathrm{mix}}(K^\ast)=\gamma^\ast$. That is, $(K^\ast,\gamma^\ast)\in\operatorname{epi}_\geq(J_{\mathrm{mix}})$.
By the inclusion $\operatorname{epi}_\geq(J_{\mathrm{mix}}) \subseteq \pi_{K,\gamma}({\mathcal{L}}_{\mathrm{lft}})$ in \eqref{eq:ECL-inclusion}, there exists $X^\ast$ such that  $(K^\ast,\gamma^\ast,X^\ast)\in{\mathcal{L}}_{\mathrm{lft}}$.
Applying the diffeomorphism ${\Phi}(K^\ast,\gamma^\ast,X^\ast)=(\gamma^\ast,X^\ast, Y^\ast)\in{\mathcal{F}}_{\mathrm{cvx}}$ then establishes solvability in this case. 

Next, consider the case where $\gamma^\ast$ is not attained in $\mathcal{K}_\beta$.
By the continuity in Lemma~\ref{thm:continuity-mix-cost}, there exists a boundary policy $K^\ast\in\partial\mathcal{K}_\beta$ with $\tilde{J}_{\mathrm{mix}}(K^\ast)\!=\!\gamma^\ast$.
Crucially, the pair $(K^\ast,\gamma^\ast)$ still admits a valid lifting to $\mathcal{L}_\mathrm{lft}$.
This follows from Corollary~\ref{coro:epi-pi-same}, where $\operatorname{epi}_\geq(\tilde{J}_{\mathrm{mix}})=\pi_{K,\gamma}(\mathcal{L}_{\mathrm{lft}})$ guarantees the existence of $X^\ast$ such that $({K}^\ast,\gamma^\ast,X^\ast)\in\mathcal{L}_\mathrm{lft}$.
Applying the diffeomorphism $\Phi(K^\ast,\gamma^\ast,X^\ast)\!=\!(\gamma^\ast, X^\ast, Y^\ast)\!\in\! \mathcal{F}_{\mathrm{cvx}}$ then shows solvability.

Combining both cases and noting that $K^\ast=Y^\ast(X^\ast)^{-1}$ since $Y^\ast=K^\ast X^\ast$, we conclude that there exists $(\gamma^\ast, X^\ast, Y^\ast)\in \mathcal{F}_{\mathrm{cvx}}$ with $K^\ast\in\mathrm{cl}(\mathcal{K}_\beta)$.
\end{proof}

Proposition~\ref{thm:solvability} guarantees solvability of the convex reformulation in \eqref{eq:mix-2ch-LMI} as well as recovery of an optimal policy, even when the optimal value is not attained in $\mathcal{K}_\beta$, in which case the policy lies on the boundary $\partial\mathcal{K}_\beta$. 

\section{Numerical Experiments} \label{section:experiments}

In this section, we present numerical experiments evaluating the performance of different methods for solving the mixed $\calH_2/\calH_\infty$ control \eqref{eq:LQR-hinf} and \eqref{eq:mix-2ch-problem} in both small and large-scale cases. 

\subsection{Numerical approaches and setup}
We compare four approaches:
\begin{enumerate}
    \item Analytical solution: For the single-channel case, we solve the Riccati equation \eqref{eq:mix-1ch-optimality-conditions-b} and obtain the optimal policy from \eqref{eq:mix-1ch-optimality-conditions-a}.

    \item Policy iteration: Inspired by the optimality conditions in Corollaries \ref{thm:mixed-2ch-optimality-conditions} and \ref{thm:mixed-1ch-optimality-conditions}, we apply iterative policy updates to solve \eqref{eq:mix-2ch-problem} and its single-channel case. The details are given below. 
    
    \item LMI-based convex optimization: Solve the convex reformulation in \eqref{eq:mix-2ch-LMI} with the conic solver \method{MOSEK} \cite{aps2019mosek}.

    \item \method{HIFOO}: A sophisticated nonsmooth optimization package (v3.501 with Hanso 2.01) for fixed-order $\mathcal{H}_2/\mathcal{H}_\infty$ controller synthesis \cite{burke2006hifoo}.

\end{enumerate}

\noindent\textbf{Policy iteration.} The method starts from a feasible policy and alternates between policy evaluation (solving a Riccati equation) and policy improvement.
It can be viewed as a fixed-point iteration for solving a stationary point $\nabla J_{\mathrm{mix}}(\cdot)=0$.
In the single-channel case with $z_2=z_\infty$, using \eqref{eq:gradient-Jtilde}, the update simplifies to:
\begin{subequations} \label{eq:PI-update}
\begin{equation} \label{eq:PI-1ch}
    K'=-R^{-1}B^\tr P_K,
\end{equation}
which is equivalent to the Gauss-Newton update: $K'=K-\eta R^{-1}\nabla J_{\mathrm{mix}}(K)\Lambda_{K}^{-1}$ with $\eta=1/2$.
This update preserves feasibility, and its convergence is established in \cite{zhang2021policy}.
For the general two-channel setting, we adopt an analogous fixed-point iteration based on \eqref{eq:gradient-J}:
\begin{equation} \label{eq:PI-2ch}
    K'=-R_2^{-1}B^\tr \Gamma_{K}(I+\beta^{-2}\alpha^2X_{K}\Gamma_{K})^{-1}.
\end{equation}
\end{subequations}
While a formal convergence analysis remains open, we conjecture that the method converges for sufficiently large $\beta$, and leave this as future work.

\noindent\textbf{Setup.} All experiments are conducted in MATLAB R2024b.
The stabilizing Riccati solution is computed using MATLAB's \method{icare}.
Policy iteration is run until convergence $\|K'-K\|<10^{-5}$.
For LMI-based methods, we use \method{MOSEK} with its default high-accuracy stopping criteria.

We compare these methods in terms of 1) the runtime complexity; 2) the square root of $J_{\mathrm{mix}}$ of the converged policy; 3) the resulting $\calH_2$ and $\calH_\infty$ norms. 

\subsection{A low-dimensional example}
We begin with a low-dimensional example (instance $0$) with a $3\times 3$ policy matrix. The problem parameters are
\begin{equation*}
A=\begin{bmatrix}
    1 & 1 & 1 \\ 0 & 1 & 0 \\ 1 & 0 & 0 
\end{bmatrix},\quad B=\begin{bmatrix}
    1 & 0 & 1 \\ 0 & 1 & 1 \\ 0 & 0 & 2
\end{bmatrix},\quad B_w=5I_3.
\end{equation*}
The performance matrices for the $\calH_2$ and $\calH_\infty$ channels are
\begin{equation*}
Q_2=\begin{bmatrix}
    1 & 1 & 1 \\ 1 & 1 & 1 \\ 1 & 1 & 1
\end{bmatrix},\quad R_2=\frac{1}{4}I_3,\quad Q_\infty=R_\infty=I_3.
\end{equation*}
For the special case $z_2=z_\infty$, we set $Q\!=\!R\!=\!I_3$.
The minimal achievable robustness level (i.e., the optimal $\calH_\infty$ norm) is $\beta^\ast \approx 5.24$.
We evaluate all methods with $\beta=6$, $14$, and $18$. 
One can verify that Assumptions~\ref{assumption:1} and \ref{assumption:positive-definiteness} are satisfied.
For reference, in the single-channel case, the optimal LQR policy achieves an $\calH_\infty$ norm $\approx 9.26$ (with the optimal LQR cost $\approx 9.92$); in the two-channel case, the corresponding values are around $19.15$ and $4.65$, respectively. Therefore, our chosen $\beta$ values are small enough so that  problem \eqref{eq:LQR-hinf} cannot be solved analytically for the two-channel case.

Table~\ref{table:low-dim-example} summarizes the performance of all four approaches.
The \texttt{ARE} row shows that solving \eqref{eq:mix-1ch-optimality-conditions-b} is the most efficient and serves as the benchmark when $z_2=z_\infty$.
From the \texttt{PI} row, we observe that the single-channel update in \eqref{eq:PI-1ch} performs reliably, requiring only about $10$ times the runtime of solving \eqref{eq:mix-1ch-optimality-conditions-b}.
In the two-channel setting, the policy iteration in \eqref{eq:PI-2ch} converges for $\beta=18$, maintaining feasibility throughout and achieving the optimal policy of \eqref{eq:mix-2ch-problem}, consistent with the LMI-based solution.
However, for smaller values of $\beta$, \eqref{eq:PI-2ch} fails to converge, an observation consistent with our conjecture.
For $\beta=6$, the iterates leave the feasible set $\mathcal{K}_\beta$ and the procedure terminates (e.g., at some iteration $t$, the $\calH_\infty$ norm of $K_t$ exceeds $6$).
For $\beta=14$, the iterates remain feasible in $\mathcal{K}_\beta$ but fail to converge, instead exhibiting a periodic behavior suggesting of a limit cycle.

From the \texttt{LMI} row, we see that the global minimum of $J_{\mathrm{mix}}$ is attained by solving the convex reformulation, with the gradient $\nabla J_{\mathrm{mix}}$ of the resulting policy nearly zero.
Overall, when successful, all three methods (\method{ARE}, \method{PI}, \method{LMI}) yield almost identical solutions.
Finally, the \method{HIFOO} solver is less reliable for smaller $\beta$, often triggering warnings or failing to return feasible solutions.
Even at $\beta=18$, its performance (especially the $\calH_\infty$ norm) deviates noticeably from the other methods.
This is expected as \method{HIFOO} relies on nonsmooth local optimization and does not guarantee global optimality.
In contrast, our result in Theorem~\ref{thm:mix-stationary-global-optimality} ensures that stationary points found via gradient-based methods are globally optimal.

\begin{table}[t]  
\setlength{\belowcaptionskip}{0pt}
\renewcommand{\arraystretch}{1.1}
\begin{center}
\caption{Comparison of different methods on Instance 0. Results for the two- and single-channel cases are shown under 2-ch and 1-ch subcolumns, respectively. A dash (–) indicates unavailable results. Bold values mark the best runtime among all methods.
}
\small 
\label{table:low-dim-example}
\begin{tabular}{ll*{6}{l}}
\toprule
\multirow{2}{*}{} & & \multicolumn{2}{c}{$\beta\!=\!6$} & \multicolumn{2}{c}{$\beta\!=\!14$} & \multicolumn{2}{c}{$\beta\!=\!18$} \\
    \cmidrule(lr){3-4} \cmidrule(lr){5-6} \cmidrule(lr){7-8}
    & & $2$-ch & $1$-ch & $2$-ch & $1$-ch & $2$-ch & $1$-ch \\
\hline
\multirow{4}{*}{\method{ARE}} & \method{time} & - & $\textbf{0.01}$ & - & $\textbf{0.01}$ & - & $\textbf{0.01}$ \\
& \method{$J_\mathrm{mix}^{1/2}$} & - & $16.1$ & - & $10.4$ & - & $10.2$ \\
& \method{$\calH_2$\!} & - & $12.7$ & - & $9.9$ & - & $9.9$ \\
& \method{$\calH_\infty$\!} & - & $5.93$ & - & $8.50$ & - & $8.79$ \\ \cline{2-8} 
\multirow{4}{*}{\method{PI}} & \method{time} & - & $0.04$ & - & $0.05$ & $\textbf{0.10}$ & $0.06$ \\
& \method{$J_\mathrm{mix}^{1/2}$} & - & $16.1$ & - & $10.4$ & $4.80$ & $10.2$ \\
& \method{$\calH_2$\!} & - & $12.7$ & - & $9.95$ & $4.67$ & $9.93$ \\
& \method{$\calH_\infty$\!} & - & $5.93$ & - & $8.50$ & $15.2$ & $8.79$ \\ \cline{2-8} 
\multirow{4}{*}{\method{LMI}} & \method{time} & $0.49$ & $0.17$ & $0.49$ & $0.18$ & $0.49$ & $0.49$ \\
& \method{$J_\mathrm{mix}^{1/2}$} & $8.12$ & $16.1$ & $4.91$ & $10.4$ & $4.80$ & $10.2$ \\
& \method{$\calH_2$\!} & $6.54$ & $12.7$ & $4.72$ & $9.95$ & $4.67$ & $9.93$ \\
& \method{$\calH_\infty$\!} & $5.96$ & $5.93$ & $13.1$ & $8.50$ & $15.2$ & $8.79$ \\ \cline{2-8} 
\multirow{3}{*}{\method{hifoo}} & \method{time} & - & - & $11.6$ & $1.59$ & $2.46$ & $2.53$ \\
& \method{$\calH_2$\!} & - & - & $>\!10^4$ & $9.92$ & $4.65$ & $9.91$ \\
& \method{$\calH_\infty$\!} & - & - & $9.69$ & $10.4$ & $17.9$ & $9.26$ \\
\toprule
\end{tabular}
\end{center}
\end{table}

\subsection{Higher-dimensional examples}

We next consider three higher-dimensional examples, with policy matrices of sizes $15\times15$, $60\times60$, and $90\times90$, referred to as instances $1$, $2$, and $3$, respectively.
The problem data for the single-channel formulations are adapted from \cite[Sec. 7.3]{zhang2021policy}.
For the two-channel formulations, we retain the same $\calH_\infty$ performance signals but specify the $\calH_2$ performance with $Q_2=I$ and $R_2=\frac{1}{4} R_\infty$.

The minimum robustness level $\beta^\ast$ for instances $1$-$3$ are $\approx 0.067$, $0.098$, and $0.096$,  respectively.
Unlike the low-dimensional example (instance $0$), our focus here is on assessing the scalability of different methods.
Accordingly, we test larger robustness levels, setting $\beta=10$, $15$, and $20$. One can verify that all assumptions are satisfied.

Table~\ref{table:high-dim-examples} summarizes the results on the higher-dimensional examples.
As shown in the \texttt{ARE} row, solving the Riccati equation remains highly efficient for the single-channel case, even at larger scales.
The \texttt{PI} row demonstrates that the updates in \eqref{eq:PI-update} remain effective as the problem size increases. Notably, the two-channel update \eqref{eq:PI-2ch} converges across all instances, with the resulting policies exhibiting nearly vanishing gradients, which empirically supports our conjecture.
This behavior suggests that \eqref{eq:PI-2ch} enjoys an implicit regularization property, maintaining feasibility and achieving convergence when $\beta$ is sufficiently large.

The \texttt{LMI} row shows that the global minimum of $J_{\mathrm{mix}}$ can be achieved by solving a convex reformulation.
However, despite yielding policies with nearly zero gradients, the LMI method scales poorly: its runtimes for instance $2$ and $3$ are significantly longer than those of the analytical and policy iteration approaches.
Overall, compared with the classical LMI-based methods, these results show that policy iteration can scale more favorably to large-scale problems.

\begin{table}[t] 
\setlength{\belowcaptionskip}{0pt}
\renewcommand{\arraystretch}{1.1}
\begin{center}
\caption{Comparison of different methods on Instance 1-3 (denoted I$_{1-3}$). Results for two- and single-channel cases are shown under 2-ch and 1-ch subcolumns, respectively. A dash (-) indicates unavailable results. Bold values mark the best runtime.
}
\small 
\label{table:high-dim-examples}
\begin{tabular}{ll*{6}{l}}
\toprule
\multirow{2}{*}{} & & \multicolumn{2}{c}{I$_1$, $\beta\!=\!10$} & \multicolumn{2}{c}{I$_2$, $\beta\!=\!15$} & \multicolumn{2}{c}{I$_3$, $\beta\!=\!20$} \\
    \cmidrule(lr){3-4} \cmidrule(lr){5-6} \cmidrule(lr){7-8}
    & & $2$-ch & $1$-ch & $2$-ch & $1$-ch & $2$-CH & $1$-CH \\
\hline
\multirow{4}{*}{\method{ARE}} & \method{time} & - & \textbf{0.02} & - & \textbf{0.03} & - & \textbf{0.05} \\
& \method{$J_{\mathrm{mix}}^{1/2}$} & - & $0.47$ & - & $1.12$ & - & $1.22$ \\
& \method{$\calH_2$\!} & - & $0.47$ & - & $1.12$ & - & $1.22$ \\
& \method{$\calH_\infty$\!} & - & 0.09 & - & 0.14 & - & 0.14 \\ \cline{2-8} 
\multirow{4}{*}{\method{PI}} & \method{time} & \textbf{0.10} & $0.06$ & \textbf{0.28} & $0.09$ & \textbf{0.61} & $0.46$ \\
& \method{$J_{\mathrm{mix}}^{1/2}$} & $0.99$ & $0.47$ & $1.99$ & $1.12$ & $2.44$ & $1.22$ \\
& \method{$\calH_2$\!} & $0.99$ & $0.47$ & $1.99$ & $1.12$ & $2.44$ & $1.22$ \\
& \method{$\calH_\infty$\!} & 1.98 & 0.09 & 7.72 & 0.14 & 9.27 & 0.14 \\ \cline{2-8} 
\multirow{4}{*}{\method{LMI}} & \method{time} & $0.27$ & $0.37$ & $11.3$ & $20.1$ & $89.7$ & $143$ \\
& \method{$J_{\mathrm{mix}}^{1/2}$} & $1.00$ & $0.47$ & $1.99$ & $1.12$ & $2.44$ & $1.22$ \\
& \method{$\calH_2$\!} & $0.99$ & $0.47$ & $1.99$ & $1.12$ & $2.44$ & $1.22$ \\
& \method{$\calH_\infty$\!} & $1.98$ & $0.09$ & $7.77$ & $0.14$ & $9.27$ & $0.13$ \\ \cline{2-8} 
\multirow{3}{*}{\method{hifoo}} & \method{time} & $1.43$ & $8.57$ & $35.6$ & $27.5$ & $262$ & $221$ \\
& \method{$\calH_2$\!} & $0.99$ & $0.47$ & $1.99$ & $1.12$ & $2.44$ & $1.22$ \\
& \method{$\calH_\infty$\!} & 2.01 & 0.09 & 8.57 & 0.14 & 10.1 & 0.14 \\
\toprule
\end{tabular}
\end{center}
\end{table}

\section{Conclusions} \label{sec:conclusions}

In this paper, we examine the nonconvex optimization landscape of mixed $\calH_2/\calH_\infty$ control.
We characterize the feasible set and cost function, showing that despite nonconvexity, every stationary point is globally optimal.
Our analysis, grounded in the \ECL{} framework, further clarifies the role of strict versus non-strict Riccati inequalities in certifying global optimality and ensuring solvability of the convex reformulation. 
Several open questions are worth further investigation.
An important one is to establish convergence guarantees for policy iteration in the general two-channel case, particularly for large $\beta$.
Another is to design principled, scalable policy optimization algorithms for mixed $\calH_2/\calH_\infty$ design with provable performance guarantees. 

\vspace{5mm}

\appendix

\noindent\textbf{Appendix}

The appendix is organized into three parts: proofs for Sections~\ref{section:preliminaries}-\ref{section:opt-landscape}, proofs for Section~\ref{section:no-spurious-stationary}, and proofs for Section~\ref{section:ECL}.

\section{Proofs in Sections~\ref{section:preliminaries} and \ref{section:opt-landscape}} \label{appendix:proof-opt-landscape}

\subsection{Proof of the cost equivalence in \eqref{eq:mix-1ch-cost-2}} \label{proof:alternative-cost-single-channel}

Let $\tilde{S}_K\!:=\!Q\!+\!K^\tr R K$.
By Lemma~\ref{lem:BRL-X}, the Riccati equations 
\begin{align}
    &A_KX_K \!+\! X_KA_K^\tr \!+\! \beta^{-2}X_K\tilde{S}_K X_K \!+\! W \!=\! 0,\label{eq:mix-1ch-Riccati-XK} \\ 
    &A_K^\tr{P}_K\!+\!{P}_KA_K\!+\!\beta^{-2}{P}_KW{P}_K\!+\!\tilde{S}_K\!=\!0 \label{eq:mix-1ch-Riccati-PK-app}
\end{align}
admit unique stabilizing solutions $X_K$ and $P_K$ respectively for any $K\in{\mathcal{K}_\beta}$.
To verify the cost equivalence in \eqref{eq:mix-1ch-cost-equivalence}, define $\tilde{X}_K\!:=\!-\beta^{-2}X_K$. Substituting into \eqref{eq:mix-1ch-Riccati-XK} gives
\begin{equation} \label{eq:mix-1ch-modify-Riccati-XK}
    A_K\tilde{X}_K + \tilde{X}_KA_K^\tr - \tilde{X}_K\tilde{S}_K \tilde{X}_K - \beta^{-2}W = 0,
\end{equation}
where the associated closed-loop matrix is
$A_K-\tilde{X}_K\tilde{S}_K=A_K+\beta^{-2}X_K\tilde{S}_K$,
which is Hurwitz by assumption. Thus, $\tilde{X}_K$ is the stabilizing solution to \eqref{eq:mix-1ch-modify-Riccati-XK}. Now consider the Hamiltonian matrices associated with \eqref{eq:mix-1ch-Riccati-PK-app} and \eqref{eq:mix-1ch-modify-Riccati-XK}:
\begin{align*}
    H(K)=\begin{bmatrix}
    A_K & \beta^{-2}W \\ -\tilde{S}_K & -A_K^\tr
    \end{bmatrix},\
    \tilde{H}(K)=\begin{bmatrix}
    A_K^\tr & -\tilde{S}_K \\ \beta^{-2}W & -A_K
    \end{bmatrix}.
\end{align*}
Since the two Hamiltonians have identical eigenvalues (as $H(K)\!=\!\tilde{H}(K)^\tr$), and their associated closed-loop matrices $A_K+\beta^{-2}WP_K$ and $A_K-\tilde{X}_K\tilde{S}_K$ are both Hurwitz, they must share the same stable eigenvalues.
Therefore, we have
    \[
    \operatorname{tr}(A_K+\beta^{-2}WP_K)=\operatorname{tr}(A_K+\beta^{-2}X_K\tilde{S}_K)
    \]
which implies $\operatorname{tr}(WP_K)=\operatorname{tr}(X_K\tilde{S}_K)$.

\subsection{Proof of Theorem~\ref{prop:closure}} \label{appendix:proof-boundary}

We prove the set inclusions in both directions. 
The first is stated below, whose proof relies critically on the global optimality property of state-feedback $\calH_\infty$ cost function.

\begin{lemma} \label{lem:subset-of-closure}
    Under Assumptions~\ref{assumption:1} and \ref{assumption:positive-definiteness}, we have 
    \[
    \left\{K\in\mathcal{K} \mid \|\mathbf{T}_{\infty}(K)\|_{\mathcal{H}_\infty}\leq\beta\right\}\subseteq \mathrm{cl}\left(\mathcal{K}_\beta\right).
    \]
\end{lemma}

\begin{proof}
    We claim that for any $K \in \mathcal{K}$ with $\|\mathbf{T}_{\infty}(K)\|_{\mathcal{H}_\infty} \le \beta$, there exists a sequence $\{K_j\} \subset \mathcal{K}_\beta$ such that $K_j \to K$, and hence $K \in \operatorname{cl}(\mathcal{K}_\beta)$.
    
    Suppose for contradiction that this is false.
    Then there exists a neighborhood $U\!\subset\!\mathcal{K}$ of $K$ such that $U\cap \mathcal{K}_\beta=\emptyset$, implying that $\|\mathbf{T}_{\infty}(K')\|_{\mathcal{H}_\infty} \geq \beta$ for all $K'\in U$.
    Thus, $K$ (with $\|\mathbf{T}_{\infty}(K)\|_{\mathcal{H}_\infty} \le \beta$) is a local minimizer of $\|\mathbf{T}_{\infty}(\cdot)\|_{\mathcal{H}_\infty}$ over $U$.
    It is clear by continuity that $\|\mathbf{T}_{\infty}(K)\|_{\mathcal{H}_\infty} = \beta$.

    However, by \cite[Corollary 4.1]{zheng2024benign}, every local minimizer of $\|\mathbf{T}_{\infty}(\cdot)\|_{\mathcal{H}_\infty}$ is a global one that attains the infimum value $\beta^\ast\!:=\!\inf_{K\in\mathcal{K}}\|\mathbf{T}_{\infty}(K)\|_{\calH_\infty}$. Therefore, we must have $\|\mathbf{T}_{\infty}(K)\|_{\mathcal{H}_\infty} \!=\! \beta^\ast$, which is a contradiction since $\beta\neq \beta^\ast$.
    Hence, such a neighborhood $U$ cannot exist.
\end{proof}

Now we state the second auxiliary lemma.
The key idea is to exclude marginally stabilizing policies by leveraging the fact that the $\mathcal{H}_\infty$ cost diverges as $K$ approaches a marginally stabilizing policy.
\begin{lemma}\label{lemma:Hinf_bigger_than_beta}
    Under Assumptions~\ref{assumption:1} and \ref{assumption:positive-definiteness}, we have 
    \[
    \mathrm{cl}\left(\mathcal{K}_\beta\right) \subseteq\left\{K\in\mathcal{K} \mid \|\mathbf{T}_{\infty}(K)\|_{\mathcal{H}_\infty}\leq\beta\right\}.
    \]
\end{lemma}

\begin{proof}
    We first prove that any $K\in \operatorname{cl}(\mathcal{K}_\beta)$ must satisfy $\|\mathbf{T}_{\infty}(K)\|_{\mathcal{H}_\infty} \leq\beta$. 
    Suppose for contradiction that there exists $K\in \operatorname{cl}(\mathcal{K}_\beta)$ with 
    $\|\mathbf{T}_{\infty}(K)\|_{\mathcal{H}_\infty} \!=\! \beta + \epsilon$ for some $\epsilon>0$.
    Then there exists a sequence $\{K_j\}\subset \mathcal{K}_\beta$ such that $K_j\to K$.
    By the continuity of $\|\mathbf{T}_{\infty}(\cdot)\|_{\mathcal{H}_\infty}$,
    there exists $N$ such that for all $j\!>\!N$, $\|\mathbf{T}_{\infty}(K_j)\|_{\mathcal{H}_\infty}\in [\beta,\beta+\epsilon)$, which contradicts $\{K_j\}\subset \mathcal{K}_\beta$.

    We next show that $\operatorname{cl}\left(\mathcal{K}_\beta\right)\subset \mathcal{K}$.
    It suffices to show that $\operatorname{cl}(\mathcal{K}_\beta) \cap  \partial\mathcal{K}=\emptyset$, where 
    $\partial\mathcal{K}$ contains all policies $K$ such that $A_K$ is marginally stable.
    The proof is by contradiction. Suppose there exists some 
    $K\in \operatorname{cl}(\mathcal{K}_\beta)\cap \partial\mathcal{K}$.
    Then there exists $\{K_j\}\subset \mathcal{K}_\beta$ with $K_j\to K$.
    Consider the sequence of transfer functions
    \[
    \mathbf{T}_\infty(K_j)
    \;=\;
    \begin{bmatrix} Q_\infty^{1/2} \\ R_\infty^{1/2}K_j \end{bmatrix}
    \,(sI-A_{K_j})^{-1} B_w ,
    \]
    where $A_{K_j}$ is stable and $\| \mathbf{T}_\infty(K_j)\|_{\mathcal{H}_\infty}< \beta$ since $K_j\in\mathcal{K}_\beta$.
    Because $K_j\to K\in\partial\mathcal{K}$, we have
    \[
    \lim_{j\to\infty}\left(A_{K_j},B_w, \begin{bmatrix} Q_\infty^{1/2} \\ R_\infty^{1/2}K_j \end{bmatrix}\right)=
    \left(A_{K},B_w,\begin{bmatrix} Q_\infty^{1/2} \\ R_\infty^{1/2}K \end{bmatrix}\right)
    \]
    where $A_K$ has at least one eigenvalue on the imaginary axis.
    Since $\left(A_{K},B_w,\begin{bmatrix} Q_\infty^{1/2} \\ R_\infty^{1/2}K \end{bmatrix}\right)$ is minimal, $\mathbf{T}_\infty(K)$ has at least one pole that lies on the imaginary axis.
    By \cite[Lemma D.1]{zheng2024benign}, it follows that
    \[
    \lim_{j\to\infty}\, \|\mathbf{T}_\infty(K_j)\|_{\mathcal{H}_\infty} = +\infty,
    \]
    which contradicts $\|\mathbf{T}_\infty(K_j)\|_{\mathcal{H}_\infty}< \beta$ for all $j$.
    Hence $\operatorname{cl}(\mathcal{K}_\beta)\cap \partial\mathcal{K}=\emptyset$, and therefore
    $\operatorname{cl}(\mathcal{K}_\beta)\subset \mathcal{K}$.
\end{proof}

Combining Lemmas~\ref{lem:subset-of-closure} and \ref{lemma:Hinf_bigger_than_beta} leads to the desired result.

\subsection{Proof of Lemma~\ref{thm:continuity-mix-cost}} \label{appendix:proof-continuity}

We aim to show that for any $K_0\in\partial \mathcal{K}_\beta$
\[
    \lim_{K\to K_0,K\in\mathcal{K}_\beta}J_{\mathrm{mix}}(K)=\tilde{J}_{\mathrm{mix}}(K_0),
\]
i.e., the extended cost function $\tilde{J}_{\mathrm{mix}}$ is continuous on the boundary $\partial \mathcal{K}_\beta$.
To this end, it suffices to establish 
\begin{equation} \label{eq:minimal-continuity}
    \lim_{K\to{K}_0,K\in\mathcal{K}_\beta}\|X_K-X_{K_0}\|=0.
\end{equation}
That is, the minimal solution $X_K$ to the Riccati equation \eqref{eq:mix-2ch-Riccati-XK} depends continuously on $K\in\mathrm{cl}(\mathcal{K}_\beta)$.

The proof relies on the following nontrivial theorem that establishes the continuous dependence of minimal solutions on the coefficients of Riccati equations.

\begin{theorem}[{\cite[Theorem 11.2.1]{lancaster1995algebraic}}] \label{thm:continuity-min-solutions}
    The minimal solution $X$ to the Riccati equation $AX + XA^\tr + XD X + C =0$
    is a continuous function of the coefficient $(A, C, D)$ on the set of all $(A,C,D)$ such that $D\succeq0$, $C=C^\tr$ and $(A,D)$ is stabilizable.
\end{theorem}


    To show \eqref{eq:minimal-continuity}, we apply Theorem~\ref{thm:continuity-min-solutions} to the Riccati equation \eqref{eq:mix-2ch-Riccati-XK} corresponding to any $K\in\mathrm{cl}(\mathcal{K}_\beta)$.
    It is not difficult to check that the coefficient matrices $A_K$, $W$, $\beta^{-2}S_K$ depend continuously on $K$ and satisfy all the conditions required by Theorem~\ref{thm:continuity-min-solutions}.
    Consequently, the minimal solution $X_K$ depends continuously on $K\in\mathrm{cl}(\mathcal{K}_\beta$), proving \eqref{eq:minimal-continuity}.

\subsection{Proof of Theorem~\ref{theorem:analytical-functions}} \label{appendix:proof-analyticity}

To prove that $J_{\mathrm{mix}}$ is real analytic on $\mathcal{K}_\beta$, it suffices to show that the stabilizing solution $X_K$ to \eqref{eq:mix-2ch-Riccati-XK} depends real analytically on $K\in\mathcal{K}_\beta$. This will be established using the following version of the implicit function theorem.

\begin{theorem}{\cite[Theorem 2.3.5]{krantz2002implicit}}
\label{theorem:implicit_function_theorem}
    Let $F: \mathbb{R}^k \times \mathbb{R}^m \rightarrow \mathbb{R}^m$ be a mapping such that each of its components, denoted by $F_j$ for $j=1, \ldots, m$, is a real analytic function. Suppose $\left(x_0, y_0\right) \in \mathbb{R}^k \times \mathbb{R}^m$ is a solution to the equation $F(x, y)=0$. Furthermore, suppose the Jacobian matrix
\begin{equation*}
\mathbf{J}_y F(x, y)=\left[\begin{array}{ccc}
\frac{\partial F_1}{\partial y_1} & \cdots & \frac{\partial F_1}{\partial y_m} \\
\vdots & \ddots & \vdots \\
\frac{\partial F_m}{\partial y_1} & \cdots & \frac{\partial F_m}{\partial y_m}
\end{array}\right]
\end{equation*}
is invertible at $(x, y)=\left(x_0, y_0\right)$. Then there exists a neighborhood $U \subseteq \mathbb{R}^k$ of $x_0$ and a mapping $\phi: U \rightarrow \mathbb{R}^m$ such that

\begin{enumerate}
    \item each component of $\phi$ is a real analytic function, and
    \item $F(x, \phi(x))=0$ for all $x \in U$.
\end{enumerate}

\end{theorem}

We now proceed to verify the conditions in Theorem~\ref{theorem:implicit_function_theorem} for the Riccati equation \eqref{eq:mix-2ch-Riccati-XK}.
By vectorizing \eqref{eq:mix-2ch-Riccati-XK}, we define the following mapping $F:\mathcal{K}_\beta\times\mathbb{R}^{n^2}\to \mathbb{R}^{n^2}$ by
\begin{align*}
    F(K,\mathrm{vex}(X)) =& \left(I_n\otimes A_K + A_K\otimes I_n\right)
    \mathrm{vec}(X) \\
    &+ \beta^{-2}\mathrm{vec}\left(XS_KX\right)  + \mathrm{vec}(W),
\end{align*}
We then compute the Jacobian of $F$ w.r.t. $\mathrm{vex}^\tr(X)$. Using standard matrix calculus and the symmetry of $X$ and $S_K$,
\begin{align*}
    &\frac{\partial \mathrm{vec}(XS_K X)}{\partial \mathrm{vec}^\tr X} \\
    =&(X^\tr \otimes I_n) \frac{\partial \mathrm{vec}(XS_K)}{\partial \mathrm{vec}^\tr X}+(I_n \otimes XS_K) \frac{\partial \mathrm{vec} X}{\partial \mathrm{vec}^\tr X} \\
    =&(X^\tr \otimes I_n)(S_K^\tr \otimes I_n)+I_n \otimes XS_K \\
    =&X^\tr S_K^\tr \otimes I_n+I_n \otimes XS_K.
\end{align*}
Denoting $\tilde{A}_K := A_K + \beta^{-2}XS_K$, it follows that
\begin{equation*} \label{eq:Jacobian-F}
    \frac{\partial F(K,\mathrm{vec}(X))}{\partial 
    \mathrm{vec}^\tr(X)}
    = I_n\otimes \tilde{A}_K
    + \tilde{A}_K\otimes I_n.
\end{equation*}
Now, observe that for any $K\in\mathcal{K}_\beta$ the matrix $\tilde{A}_K$
is Hurwitz (and hence invertible) when $X$ is the stabilizing solution to \eqref{eq:mix-2ch-Riccati-XK}.
Thus, for any pair $(K_0,\mathrm{vec}(X_0))$ satisfying $F(K_0,\mathrm{vec}(X_0))=0$, where $X_0$ is the stabilizing solution associated with $K_0\in\mathcal{K}_\beta$, the above Jacobian 
is invertible. 
By Theorem~\ref{theorem:implicit_function_theorem}, it follows that $X_K$ is real analytic in a neighborhood of $K_0\in\mathcal{K}_\beta$.
Since $K_0$ is arbitrary, we conclude that $X_K$ is real analytic on the entire set $\mathcal{K}_\beta$. \hfill $\qed$

\subsection{Proof of Lemma~\ref{theorem:policy-gradient-mixed}} \label{app:proof-mix-grad-2ch}
We begin with an auxiliary lemma that characterizes the derivative of the stabilizing solution to a Riccati equation.
The key idea is to linearize the Riccati equation around the solution, yielding a Lyapunov equation for the derivative.


\begin{lemma} \label{lem:proof-mix-grad-2ch-prep}
    Let $M:(-\delta,\delta)\to\mathbb{R}^{n\times n}$, $G:(-\delta,\delta)\to\mathbb{S}^{n}_+$, and $H:(-\delta,\delta)\to\mathbb{S}^{n}_+$ be real analytic matrix-valued functions for some $\delta>0$.
    Suppose that the Riccati equation
    \begin{equation} \label{eq:proof-grad-2ch-Riccati}
        \mathcal{R}(X_t,t):=M_tX_t+X_tM_t^\tr+G_t+X_tH_tX_t=0
    \end{equation}
    admits a unique \textit{stabilizing} solution $X_t$ for all $t\in(-\delta,\delta)$, i.e.,  $M_t+X_tH_t$ is Hurwitz.
    Then $X_t$ is real analytic over $t\in(\delta,\delta)$, and its derivative at $t=0$, denoted by $X_0'$, satisfies the following Lyapunov equation
    \begin{equation}\label{eq:proof-mix-2ch-grad-Lyapunov}
    \begin{aligned}
        (M_0&+X_0H_0)X_0'+X'_0(M_0+X_0H_0)^\tr\\
        &+(M'_0X_0+X_0M'^\tr_0+G'_0+X_0H'_0X_0)=0.
    \end{aligned}
    \end{equation}
\end{lemma}

\begin{proof}
    From the proof of Theorem~\ref{theorem:analytical-functions}, the stabilizing solution $X_t$ is real analytic in $t$.
    Since $M_t,G_t,H_t$, and $X_t$ are real analytic, each admits a Taylor expansion near $t=0$:
    \begin{equation} \label{eq:taylor-expansion}
        \begin{aligned}
        M_t&=M_0+tM'_0+o(t),\ G_t=G_0+tG'_0+o(t),\\
        H_t&=H_0+tH'_0+o(t),\ X_t=X_0+tX'_0+o(t).
    \end{aligned}
    \end{equation}
    Substituting \eqref{eq:taylor-expansion} into the Ricatti equation \eqref{eq:proof-grad-2ch-Riccati} and using $\mathcal{R}(X_0,0)=0$ yields $\mathcal{M}t+o(t)=0$, where $\mathcal{M}$ denotes the LHS of \eqref{eq:proof-mix-2ch-grad-Lyapunov}.
    Since this equality holds for all small $t$, the coefficient of $t$ must vanish, giving \eqref{eq:proof-mix-2ch-grad-Lyapunov}.
    \end{proof}

We now derive the gradient formula of $J_{\mathrm{mix}}$.
%
%
Consider an arbitrary direction $\Delta\in\mathbb{R}^{m\times n}$. For all sufficiently small $t>0$ such that $K+t\Delta\in\mathcal{K}_\beta$, define $M_t:=A+B(K+t\Delta)$, $G_t:=W$, $H_t=\beta^{-2}(Q_\infty+(K+t\Delta)^\tr R_\infty (K+t\Delta))$, and let $X_t$ be the stabilizing solution to $M_tX_t+X_tM_t^\tr+G_t+X_tH_tX_t=0$, whose existence is guaranteed by Lemma~\ref{lem:BRL-X}.
At $t=0$, we have
\begin{align} \label{eq:proof-mix-grad-2ch-notation}
    \begin{split}
        &M_0=A_K,\ M'_0=B\Delta,\ G'_0=0,\ H_0=\beta^{-2}S_K \\
        &H'_0=\beta^{-2}(K^\tr R_\infty \Delta+\Delta^\tr R_\infty K),\ X_0=X_K.
    \end{split}  
\end{align}
By Lemma~\ref{lem:proof-mix-grad-2ch-prep}, $X_t$ admits the first-order expansion:
\begin{equation} \label{eq:taylor-XKdelta}
    X_t=X_K+tX'_0+o(t).
\end{equation}
By \eqref{eq:proof-mix-2ch-grad-Lyapunov} and \eqref{eq:proof-mix-grad-2ch-notation}, the derivative $X'_0$ satisfies
\begin{align*} \label{eq:mix-grad-lyap-2ch}
        \tilde{A}_K X'_0&+X'_0 \tilde{A}_K^\tr + B\Delta X_K+X_K\Delta^\tr B^\tr\\
        &+\beta^{-2}X_K(K^\tr R_\infty\Delta+\Delta^\tr R_\infty K)X_K=0,
\end{align*}
where $\tilde{A}_K:=A_K+\beta^{-2}X_KS_K$. The solution is given by
\begin{equation} \label{eq:mix-grad-lyap-sol-2ch}
\begin{split}
    X'_0 = &\int_0^\infty e^{\tilde{A}_K s} \Big( B\Delta X_K + X_K \Delta^\tr B^\tr \\
    +& \beta^{-2} X_K (K^\tr R_\infty \Delta + \Delta^\tr R_\infty K) X_K \Big) e^{\tilde{A}_K^\tr s} \, ds.
\end{split}
\end{equation}
Now using \eqref{eq:taylor-XKdelta}, we can expand the cost as
    \begin{align*}
        &J_{\mathrm{mix}}(K+t\Delta)= o(t) + J_{\mathrm{mix}}(K) \\
        &+ t\cdot \Tr\left( (K^\tr R_2 \Delta + \Delta^\tr R_2 K) X_K + (Q_2 + K^\tr R_2 K) X'_0 \right). 
    \end{align*}
Hence, the directional derivative is
\begin{equation*} \label{eq:proof-mix-grad-2ch-derivative}
    \begin{aligned} 
    \left.\frac{d J_{\mathrm{mix}}(K+t\Delta)}{dt}\right\rvert_{t=0}
    =&\Tr((K^\tr R_2\Delta\!+\!\Delta^\tr R_2K)X_K)\\
    &+\Tr((Q_2\!+\!K^\tr R_2K)X'_0).
    \end{aligned}
\end{equation*}
Substituting $X_0'$ from \eqref{eq:mix-grad-lyap-sol-2ch} gives the second trace term as
\[
\Tr\left(2X_K(\Gamma_KB+\beta^{-2}\Gamma_KX_KK^\tr R_\infty)\Delta\right)
\]
    where $\Gamma_K$ solves the Lyapunov equation $\tilde{A}_K^\tr \Gamma_K+\Gamma_K\tilde{A}_K+Q_2+K^\tr R_2K=0$.
    Finally, using the identity 
    \[
    \left.\frac{dJ_{\mathrm{mix}}(K+t\Delta)}{dt}\right\rvert_{t=0}=\Tr(\nabla J_{\mathrm{mix}}(K)^\tr\Delta),
    \]
    we identify the gradient formula in \eqref{eq:gradient-J}. \hfill $\qed$

\section{Proofs in Section~\ref{section:no-spurious-stationary}} \label{appendix:proof-no-spurious-stationary}

\subsection{Proof of Corollary~\ref{thm:mixed-1ch-optimality-conditions}} \label{appendix:proof-optimality-signle-channel}

\begin{proof}
    The proof is similar to that of Corollary~\ref{thm:mixed-2ch-optimality-conditions}. We prove both directions of the equivalence.
    
    $\Rightarrow$. 
    Suppose $K\in{\mathcal{K}_\beta}$ is a global minimizer of \eqref{eq:mix-2ch-problem}. Then, its policy gradient given by \eqref{eq:gradient-Jtilde} must vanish, i.e., $\nabla{J_{\mathrm{mix}}(K)}=2(RK+B^\tr {P}_K)\Lambda_K=0$, where ${P}_K\succeq0$ is the unique stabilizing solution to the Riccati equation \eqref{eq:mix-1ch-Riccati-PK}
    \[
    A_K^\tr{P}_K+{P}_KA_K+\beta^{-2}{P}_KW{P}_K+Q+K^\tr R K=0,
    \]
    and $\Lambda_K$ is the unique solution to the Lyapunov equation \eqref{eq:mix-1ch-grad-Lyapunov-eq} associated with $K$.
    Since $W\succ0$, the Lyapunov equation implies $\Lambda_K \succ 0$, and thus we have $K=-R^{-1}B^\tr P_K$. Now letting $P:=P_K$, we immediately obtain \eqref{eq:mix-1ch-optimality-conditions-a}. Substituting this into the Riccati equation \eqref{eq:mix-1ch-Riccati-PK} gives
    \begin{align*}
        (A-BR^{-1}B^\tr {P})^\tr& {P} + {P}(A-BR^{-1}B^\tr {P})\\
        +&\beta^{-2}P W P+Q+ {P}BR^{-1}B^\tr {P} =0, 
    \end{align*}
    which is identical to \eqref{eq:mix-1ch-optimality-conditions-b}.
    Finally, since $P$ is the stabilizing solution to \eqref{eq:mix-1ch-Riccati-PK}, the matrix
    $
    A_K+\beta^{-2}WP = A+(\beta^{-2}W-BR^{-1}B^\tr)P
    $
    is Hurwitz. 

    $\Leftarrow$. 
    Now suppose $K=-R^{-1}B^\tr P$, where $P\succeq0$ is the unique stabilizing solution to the Riccati equation \eqref{eq:mix-1ch-optimality-conditions-b} such that $A+(\beta^{-2}W-BR^{-1}B^\tr)P$ is Hurwitz. First, substituting $K=-R^{-1}B^\tr P$ into the following Riccati equation (which matches the form of \eqref{eq:mix-1ch-Riccati-PK})
    \[
    A_K^\tr{P}+{P}A_K+\beta^{-2}{P}W{P}+Q+K^\tr R K=0,
    \]
    we recover exactly \eqref{eq:mix-1ch-optimality-conditions-b}. This confirms that $P$ satisfies the Riccati equation \eqref{eq:mix-1ch-Riccati-PK} associated with $K$. Next, we verify that $P$ is the stabilizing solution to \eqref{eq:mix-1ch-Riccati-PK}. Observe that
    \[
    A_K+\beta^{-2}WP=A+(\beta^{-2}W-BR^{-1}B^\tr)P.
    \]
    By assumption, this matrix is Hurwitz, which confirms that $P$ is a stabilizing solution to \eqref{eq:mix-1ch-Riccati-PK}. Therefore, it follows from Lemma~\ref{lem:BRL-X} that $K\in{\mathcal{K}_\beta}$. Finally, we recognize from the gradient formula \eqref{eq:gradient-Jtilde} that $\nabla{J_{\mathrm{mix}}(K)}=2(RK+B^\tr {P})\Lambda_K=0$, so $K$ is a stationary point. We then conclude by Theorem~\ref{thm:mix-stationary-global-optimality} that $K$ is globally optimal for \eqref{eq:mix-2ch-problem}. 
\end{proof}

\subsection{Proof of Theorem~\ref{prop:existence-2ch-large-beta}} \label{appendix:proof-2ch-existence-stationary}

\begin{proof}
Our proof is based on the implicit function theorem (Theorem~\ref{theorem:implicit_function_theorem}).
Let $\epsilon:=\beta^{-1}$.
From optimality conditions similar to \eqref{eq:mix-2ch-optimality-conditions} (without assuming that $R_\infty =\alpha^2 R_2$),
we define the following vectorized function
\begin{align*}
&F(K,X,\Gamma,\epsilon)
=
\begin{bmatrix}
    F_K(K,X,\Gamma,\epsilon)\\
    F_X(K,X,\Gamma,\epsilon)\\
    F_\Gamma(K,X,\Gamma,\epsilon)
\end{bmatrix}\\
=&\begin{bmatrix}
\operatorname{vec}\left(  R_2 K+ B^\tr \Gamma+\epsilon^2 R_\infty K X\Gamma\right)\\
\operatorname{vec}\left(  
A_KX + XA_K^\tr  + \epsilon^{2}XS_K X + W \right) \\
\operatorname{vec}\left(  
    \tilde{A}_K^\tr \Gamma+\Gamma \tilde{A}_K + Q_2+K^\tr R_2K 
    \right)
  \end{bmatrix}
\end{align*}
where
$\tilde{A}_K = A_K+ \epsilon^2 XS_K$ is Hurwitz.
This equation characterizes the minimizer of $J_\mathrm{mix}$.
Since $\epsilon_0=0$ gives the LQR case,
we know that there exists some $(X_0,\Gamma_0,K_0)$ satisfying
\begin{equation*}
F(X_0,\Gamma_0,K_0,\epsilon_0)=0
\end{equation*}
and $A_{K_0}:=A+BK_0$ is guaranteed to be stable.
Then, it is not very difficult to compute
the Jacobian $\mathbf{J}_{(K,X,\Gamma)}F$ of $F$ at $(K_0,X_0,\Gamma_0,0)$
\begin{equation*}
\mathbf{J}_{(K,X,\Gamma)}F(K_0,X_0,\Gamma_0,0)
= 
\begin{bmatrix}
I\otimes R_2 &  [\begin{array}{cc}
0 & (I\otimes B^\tr) 
\end{array}]\\
\left[\begin{array}{c}
     * \\
     0 
\end{array}\right]&
\left[\begin{array}{cc}
H_1 &  0 \\
     0& H_2
\end{array}\right]
\end{bmatrix},
\end{equation*}
where
$H_1$ and $H_2$ are the following invertible matrices
\begin{align*}
  H_1 &=  I\otimes A_{K_0}+A_{K_0}\otimes I,\\
  H_2 &= I\otimes A_{K_0}^\tr + A_{K_0}^\tr \otimes I.
\end{align*}
From the Schur complement,
since
\begin{align*}
    I\otimes R_2
    - \begin{bmatrix}
        0& (I\otimes B^\tr)
    \end{bmatrix}
    \begin{bmatrix}
        H_1^{-1} & 0\\
        0 & H_2^{-1}
    \end{bmatrix}
    \begin{bmatrix}
    *\\0
    \end{bmatrix}=
   I\otimes R_2
\end{align*}
is invertible, 
so is $\mathbf{J}_{(K,X,\Gamma)} F(K_0,X_0,\Gamma_0,0)$. 
Therefore, the desired result follows from Theorem~\ref{theorem:implicit_function_theorem}.
In particular, there exists a neighborhood $U\subset \mathbb{R}$ of $\epsilon_0=0$ and 
a real analytic mapping $\phi:U\to \mathbb{R}^{mn+2n^2}$
such that 
$$
F(\phi(\epsilon),\epsilon)=0,\quad
\forall \epsilon\in U,
$$
and $A_K$ is Hurwitz.
Note that the stability of $A_K$ is guaranteed for any sufficiently small $\epsilon>0$.
Therefore, a solution
to \eqref{eq:mix-2ch-optimality-conditions}
exists for any sufficiently small $\epsilon\in U$, i.e.,
for any sufficiently large $\beta> (\sup_{\epsilon\in U}\epsilon)^{-1}$.
\end{proof}

\section{Proofs in Section~\ref{section:ECL}} \label{appendix:proof-ECL}

\subsection{Proof of Lemma~\ref{lem:inequalities-for-J}} \label{subsec:proof-strict-lmis}

\begin{proof}
    $\Rightarrow$.
We first prove this simpler direction. Suppose $K\in\mathrm{cl}(\mathcal{K}_\beta)$ and $\tilde{J}_{\mathrm{mix}}(K)\leq\gamma$. Let $X:=X_K\succ0$ be the minimal solution to \eqref{eq:mix-2ch-Riccati-XK}. It follows that 
\[
\Tr(Q_2X)+\Tr(R_2 K^\tr R_2 K)=J_{\mathrm{mix}}(K)\leq\gamma.
\]

\vspace{2mm}

$\Leftarrow$.
We first show that $A_K$ is stable. This follows from the assumption $W\succ0$ and the existence of $X\succ0$ satisfying
\[
A_KX + XA_K^\tr + \beta^{-2}XS_K X + W \preceq 0.
\]
Then, applying the non-strict bounded real lemma \cite[Lemma 3.2 (2)]{zheng2023benign} leads to $\|\mathbf{T}_{\infty}(K)\|_{\calH_\infty}\leq \beta$, implying that $K\in\mathrm{cl}(\mathcal{K}_\beta)$.
We next verify $\gamma\geq \tilde{J}_{\mathrm{mix}}(K)$.
Let $X_K$ denote the minimal solution to \eqref{eq:mix-2ch-Riccati-XK}. We observe
\begin{align*}
    \gamma &\geq \Tr(Q_2X)+\Tr(R_2 KXK^\tr)\\
           &\geq \Tr(Q_2X_K)+\Tr(R_2 KX_KK^\tr)=\tilde{J}_{\mathrm{mix}}(K),
\end{align*}
where the first inequality is by assumption, and the second utilizes the minimal solution property $X\succeq X_K$.
\end{proof}


\subsection{Proof of Theorem~\ref{proposition:mixed-lifting}} \label{subsection:proof-mixed-lifting}

The proof includes two parts:
\begin{itemize}
    \item The mapping $\Phi$ given by \eqref{eq:mapping-mixed-state-feedback} is a $C^2$ (and in fact $C^\infty$) diffeomorphism from $\mathcal{L}_{\mathrm{lft}}$ to $\mathcal{F}_{\mathrm{cvx}}$. This is proved in Appendix~\ref{subsubsec:proof-diffeo}.
    \item The inclusion $\operatorname{epi}_\geq(J_{\mathrm{mix}}) \subseteq \pi_{K,\gamma}(\mathcal{L}_{\mathrm{lft}})
        = \operatorname{cl}\operatorname{epi}_\geq(J_{\mathrm{mix}})$ given in \eqref{eq:ECL-inclusion} holds. This is proved in Appendix~\ref{subsubsec:proof-inclusion}.
\end{itemize}

\subsubsection{Proof of the Diffeomorphism from $\mathcal{L}_{\mathrm{lft}}$ to $\mathcal{F}_{\mathrm{cvx}}$}\label{subsubsec:proof-diffeo}
We first provide some auxiliary notations. Denote
\begin{align*}
\mathcal{L}_{\mathrm{lft}}^0
& =\left\{ (K,\gamma,X) \left|\,
       K\in \mathbb{R}^{m\times n}, \;\gamma\in\mathbb{R},\;
         X \succ 0
        \right.\right\}, \\
\mathcal{F}_{\mathrm{cvx}}^0
& =\left\{
\left(\gamma, X, Y \right) \left|\,
\gamma\in\mathbb{R},\;
X \succ 0,\;
Y\in \mathbb{R}^{m\times n}
        \right.\right\}.
\end{align*}
Evidently, $\mathcal{L}_{\mathrm{lft}}^0$ and $\mathcal{F}_{\mathrm{cvx}}^0$ can be identified with certain open subsets of some Euclidean spaces, and $\mathcal{F}_{\mathrm{cvx}}^0$ is convex. Also,
\begin{equation} \label{eq:lifted-set-mixed-state-feedback-new}
\scalebox{0.81}{$
    \mathcal{L}_{\mathrm{lft}} 
    \!=\!  \left\{(K,\gamma,X)\!\in\!\mathcal{L}_{\mathrm{lft}}^0\left| 
    \begin{aligned}
    \gamma\geq\Tr(Q_2X)+\Tr(R_2 KXK^\tr)&, \\
    A_KX \!+\! XA_K^\tr \!+\! W \!+\! \beta^{-2}XS_K X \!\preceq \!0&
    \end{aligned} \right.
    \right\},$}
\end{equation}

\begin{equation} \label{eq:convex-set-mixed-state-feedback-new}
\scalebox{0.81}{$
\mathcal{F}_{\mathrm{cvx}}
\!=\!
\left\{
(\gamma, X, Y)\!\in\!\mathcal{F}_{\mathrm{cvx}}^0\left| 
\begin{aligned}
    \gamma\geq\Tr(Q_2X)+\Tr(R_2 YX^{-1}Y^\tr)&,\\
    \mathbb{F}_{\mathrm{cvx}}\preceq0&
\end{aligned}\right.
\right\}.$}
\end{equation}
We can now extend the domain of $\Phi$ in  \eqref{eq:mapping-mixed-state-feedback} so that it is now defined on the larger set $\mathcal{L}_{\mathrm{lft}}^0$:
\begin{equation}\label{eq:extend-def-Phi}
    \Phi(K,\gamma,X)=\left(\gamma,X,KX\right),\quad \forall (K,\gamma,X)\in\mathcal{L}_{\mathrm{lft}}^0.
\end{equation}

Our subsequent proof consists of showing the following:
\begin{enumerate}
\item $\Phi$ is a $C^2$ diffeomorphism from $\mathcal{L}_{\mathrm{lft}}^0$ to $\mathcal{F}_{\mathrm{lft}}^0$.

\item For any $(K,\gamma,X)\!\in\!\mathcal{L}_{\mathrm{lft}}^0$, we have $(K,\gamma,X)\!\in\!\mathcal{L}_{\mathrm{lft}}$ if and only if 
$\Phi(K,\gamma,X)\!\in\!\mathcal{F}_{\mathrm{cvx}}$.
\end{enumerate}

\vspace{2mm}
\noindent\textbf{Part I.} Consider $\Phi$ in \eqref{eq:extend-def-Phi} and define the mapping 
\begin{equation} \label{eq:inverse-Phi-mix}
    \Psi(\gamma, X, Y) = (YX^{-1}, \gamma, X), \quad \forall (\gamma, X, Y) \in \mathcal{F}_\mathrm{cvx}^0.
\end{equation}
It is immediate to verify that  $\Phi(\mathcal{L}_{\mathrm{lft}}^0)\subseteq \mathcal{F}_{\mathrm{cvx}}^0$ and $\Psi(\mathcal{F}_\mathrm{cvx}^0) \subseteq \mathcal{L}_\mathrm{lft}^0$. We are now allowed to form the compositions $\Phi \circ \Psi:\mathcal{L}_\mathrm{lft}^0\to\mathcal{L}_\mathrm{lft}^0$ and $\Psi \circ \Phi:\mathcal{F}_\mathrm{cvx}^0\to\mathcal{F}_\mathrm{cvx}^0$. It is also straightforward to see that these two compositions are the identity functions on their domains.

Summarizing the previous results, we can conclude that $\Phi$ is a bijection from $\mathcal{L}_{\mathrm{lft}}^0$ to $\mathcal{F}_\mathrm{cvx}^0$ and $\Psi$ is its inverse. Finally, $\Phi$ and $\Psi$ are in fact both $C^\infty$ functions since each element of $\Phi$ or $\Psi$ is a rational function.

\vspace{3mm}
\noindent\textbf{Part II.}
Consider the mapping $\Phi$ in \eqref{eq:extend-def-Phi}. We will show that $(K,\gamma,X)\in\mathcal{L}_{\mathrm{lft}}$ if and only if 
$\Phi(K,\gamma,X)\in\mathcal{F}_{\mathrm{cvx}}$.
Let $(K,\gamma,X)\in\mathcal{L}_{\mathrm{lft}}$ be arbitrary, and let $Y\!=\!KX$ so that $
\Phi(K,\gamma,X)=(\gamma,X,Y)$. By the definition of $\mathcal{L}_{\mathrm{lft}}$ in \eqref{eq:lifted-set-mixed-state-feedback-new}, $(K,\gamma,X)$ satisfies the following inequalities:
\begin{align}
    &X\succ0,\ \gamma\geq\Tr(Q_2X)+\Tr(R_2 KXK^\tr), \nonumber \\
    &A_KX + XA_K^\tr + W + \beta^{-2}XS_K X \!\preceq \!0 \label{eq:ARE-Llft-proof}.
\end{align}
Since $X\succ0$, we can express $K=YX^{-1}$.
By Schur complement and using $Y=KX$, the Riccati inequality in \eqref{eq:ARE-Llft-proof} is equivalent to $\mathbb{F}_{\mathrm{cvx}}\preceq0$.
Comparing the obtained properties of $(\gamma,X,Y)$ with the definition of $\mathcal{F}_\mathrm{cvx}$ in \eqref{eq:convex-set-mixed-state-feedback-new}, we confirm that $(\gamma,X,Y) \in \mathcal{F}_{\mathrm{cvx}}$. 
By reversing the derivation above, it is not difficult to prove the other direction.

\subsubsection{Proof of \eqref{eq:ECL-inclusion}} \label{subsubsec:proof-inclusion}
We first show the inclusion 
$
\operatorname{epi}_{\geq} (J_{\mathrm{mix}})\subseteq \pi_{K,\gamma}(\mathcal{L}_{\mathrm{lft}}),
$
which follows immediately from Corollary~\ref{coro:epi-pi-same} and the fact that
$\operatorname{epi}_{\geq}(J_{\mathrm{mix}}) \subseteq \operatorname{epi}_{\geq}(\tilde{J}_{\mathrm{mix}})$.
Next, the set identity
\begin{equation} \label{eq:set-identity}
    \pi_{K,\gamma}(\mathcal{L}_{\mathrm{lft}})=\mathrm{cl}\,\operatorname{epi}_{\geq} (J_{\mathrm{mix}}).
\end{equation}
will be established through the following lemma.

\begin{lemma} \label{lem:epi-clepi-same}
Under Assumptions~\ref{assumption:1}~and~\ref{assumption:positive-definiteness}, we have $\operatorname{epi}_\geq(\tilde{J}_{\mathrm{mix}}) = \operatorname{cl}\operatorname{epi}_\geq(J_{\mathrm{mix}})$.
\end{lemma}

\begin{proof}
    $\subseteq$. Let $(K,\gamma)\in\operatorname{epi}_\geq(\tilde{J}_{\mathrm{mix}})$, i.e., $K\in\mathrm{cl}(\mathcal{K}_\beta)$ and $\tilde{J}_{\mathrm{mix}}(K)=\Tr((Q_2+K^\tr R_2 K) X_K)\leq\gamma$.
    We claim that there exists a sequence in $\operatorname{epi}_\geq(J_{\mathrm{mix}})$ converging to $(K,\gamma)$.
    Since $K\in\mathrm{cl}(\mathcal{K}_\beta)$, there exists a sequence $\{K_n\}\subset \mathcal{K}_\beta$ such that $K_n\to K$.
    Since $\tilde{J}_{\mathrm{mix}}$ is continuous on $\mathrm{cl}(\mathcal{K}_\beta)$, we have $\tilde{J}_{\mathrm{mix}}(K_n)\to \tilde{J}_{\mathrm{mix}}(K)$. Pick $\gamma_n=\gamma+\frac{1}{n}$. Then $(K_n,\gamma_n)\in\operatorname{epi}_\geq(J_{\mathrm{mix}})$ and $(K_n,\gamma_n)\to(K,\gamma)$. Therefore, $(K,\gamma)\in\operatorname{cl}\operatorname{epi}_\geq(J_{\mathrm{mix}})$.

    \vspace{2mm}
    
    $\supseteq$. Let $(K,\gamma)\in\operatorname{cl}\operatorname{epi}_\geq(J_{\mathrm{mix}})$, i.e., there exists a sequence $(K_n,\gamma_n)\in\operatorname{epi}_\geq(J_{\mathrm{mix}})$ such that $(K_n,\gamma_n)\to(K,\gamma)$.
    We have $K_n\in\mathcal{K}_\beta$, ${J}_{\mathrm{mix}}(K_n)\leq \gamma_n$ and $K_n\to K\in\mathrm{cl}(\mathcal{K}_\beta)$.
    By the continuity of $\tilde{J}_{\mathrm{mix}}$ on $\mathrm{cl}(\mathcal{K}_\beta)$, taking limits on both sides of $J_{\mathrm{mix}}(K_n)\leq\gamma_n$ yields $\tilde{J}_{\mathrm{mix}}(K)\leq\gamma$.
    Therefore, $(K,\gamma)\in\operatorname{epi}_\geq(\tilde{J}_{\mathrm{mix}})$.
\end{proof}

Combining Lemma~\ref{lem:epi-clepi-same} and Corollary~\ref{coro:epi-pi-same} yields \eqref{eq:set-identity}.

\vspace{3mm}

\balance

{\small
\begin{spacing}{0.9}
\bibliographystyle{IEEEtran}
\bibliography{ref_etal.bib}
\end{spacing}
}

\end{document}